\documentclass[11pt]{elsarticle}

\usepackage[ddmmyyyy]{datetime}
\usepackage{geometry}
\RequirePackage[colorlinks,citecolor=blue,urlcolor=blue]{hyperref}
\usepackage{float}
\usepackage{amsfonts}
\usepackage{mathrsfs}
\usepackage{amssymb}
\usepackage{amsmath}
\usepackage{amsthm}
\usepackage{appendix}
\usepackage{xcolor}
\usepackage{graphicx}
\usepackage[textfont=scriptsize,labelfont=bf]{caption}
\usepackage[textfont=scriptsize,labelfont=scriptsize]{subcaption}
\usepackage{algpseudocode}
\usepackage[T1]{fontenc}
\usepackage{fSPDE}
\usepackage{fSPDEmicros}

\makeatletter
\def\ps@pprintTitle{%
 \let\@oddhead\@empty
 \let\@evenhead\@empty
 \def\@oddfoot{}%
 \let\@evenfoot\@oddfoot}
\makeatother

\begin{document}

\begin{frontmatter}

\title{On Approximation for Fractional Stochastic Partial Differential Equations on the Sphere\tnoteref{tifn}}
\tnotetext[tifn]{This research was supported under the Australian Research Council's \emph{Discovery Project} DP160101366.}

\author[addqut,addxtu]{Vo V. Anh}
\ead{v.anh@qut.edu.au}
\author[addltu]{Philip Broadbridge}
\ead{P.Broadbridge@latrobe.edu.au}
\author[addltu]{Andriy Olenko}
\ead{A.Olenko@latrobe.edu.au}
\author[addltu,addunsw]{Yu Guang Wang\corref{corau}}
\ead{y.wang@latrobe.edu.au}

\cortext[corau]{Corresponding author.}


\address[addqut]{School of Mathematical Sciences, Queensland University of Technology, Brisbane, QLD, 4000, Australia}
\address[addxtu]{School of Mathematics and Computational Science, Xiangtan University, Hunan, 411105, China}
\address[addltu]{Department of Mathematics and Statistics, La Trobe University, Melbourne, VIC, 3086, Australia}
\address[addunsw]{School of Mathematics and Statistics, The University of New South Wales, Sydney, NSW, 2052, Australia}

\begin{abstract}
This paper gives the exact solution in terms of the Karhunen-Lo\`{e}ve expansion to a fractional stochastic partial differential equation on the unit sphere $\sph{2}\subset \Rd[3]$ with fractional Brownian motion as driving noise and with random initial condition given by a fractional stochastic Cauchy problem. A numerical approximation to the solution is given by truncating the Karhunen-Lo\`{e}ve expansion. We show the convergence rates of the truncation errors in degree and the mean square approximation errors in time. Numerical examples using an isotropic Gaussian random field as initial condition and simulations of evolution of cosmic microwave background (CMB) are given to illustrate the theoretical results.
\end{abstract}


\end{frontmatter}

\section{Introduction}\label{sec:intro}
Fractional stochastic partial differential equations (fractional SPDEs) on the unit sphere $\sph{2}$ in $\Rd[3]$ have numerous applications in environmental modelling and astrophysics, see \citep{AnKeLeRu2008,Brillinger1997,CaSt2013,Dodelson2003,Durrer2008,Hristopulos2003,LaGu1999,PiScWh2000,Planck2016I,RuReMe2013,Stein2007,StChAn2013}.
One of the merits of fractional SPDEs is that they can be used to maintain long range dependence in evolutions of complex systems \citep{AnLeRu2016,BeDuKa2015,HuLiNu2016,Inahama2013,Lyons1998}, such as climate change models and the density fluctuations in the primordial universe as inferred from the cosmic microwave background (CMB). 

In this paper, we give the exact and approximate solutions of the fractional SPDE on $\sph{2}$
\begin{equation}\label{eq:fSPDE}
	\IntD{\sol(t,\PT{x})} + \psi(-\LBo) \sol(t,\PT{x}) = \IntD{\fBmsph(t,\PT{x})},\quad t\ge0,\;\PT{x}\in\sph{2}.
\end{equation}
Here, for $\alpha\ge0$, $\gamma>0$,
the \emph{fractional diffusion operator}
\begin{equation}\label{eq:fLBo}
  \psi(-\LBo):=\frLBo
\end{equation}
is given in terms of Laplace-Beltrami operator $\LBo$ on $\sph{2}$ with
\begin{equation}\label{eq:psi}
	\psi(t):=t^{\alpha/2}(1+t)^{\gamma/2},\quad t\in\Rplus.
\end{equation}
The noise in \eqref{eq:fSPDE} is modelled by a \emph{fractional Brownian motion} (fBm) $\fBmsph(t,\PT{x})$ on $\sph{2}$ with Hurst index $\hurst\in[1/2,1)$ and variances $\vfBm$ at $t=1$. When $H=1/2$, $\fBmsph(t,\PT{x})$ reduces to the Brownian motion on $\sph{2}$.

The equation \eqref{eq:fSPDE} is solved under the initial condition $\sol(0,\PT{x})=\solC(t_{0},\PT{x})$, where $\solC(t_{0},\PT{x})$, $t_{0}\ge0$, is a \emph{random field} on the sphere $\sph{2}$, which is the solution of the fractional stochastic Cauchy problem at time $t_{0}$:
\begin{equation}\label{eq:fSCauchy}
	\begin{array}{ll}
	 \displaystyle\pdiff{\solC(t,\PT{x})}{t} + \psi(-\LBo) \solC(t,\PT{x}) =0\\[3mm]
	 \solC(0,\PT{x}) = \RF_{0}(\PT{x}),
	 \end{array}
\end{equation}
where $\RF_{0}$ is a (strongly) \emph{isotropic Gaussian random field} on $\sph{2}$, see Section~\ref{sec:fSCauchy}. For simplicity, we will skip the variable $\PT{x}$ if there is no confusion.

The fractional diffusion operator $\psi(-\LBo)$ on $\sph{2}$ in \eqref{eq:fSCauchy} and \eqref{eq:fLBo} is the counterpart to that in $\Rd[3]$. We recall that the operator $\mathcal{A}:=-\left(-\Delta\right)^{\alpha/2}\left( I-\Delta \right)^{\gamma/2}$, which is the inverse of the composition of the Riesz potential $\left(-\Delta\right)^{-\alpha/2}$, $\alpha \in (0,2]$, defined by the kernel 
\begin{equation*}
J_{\alpha }\left( x\right) =\frac{\Gamma \left(3/2-\alpha \right) }{\pi^{3/2}4^{\alpha }\Gamma \left( \alpha \right) }\left\vert x\right\vert^{2\alpha -3},\quad x\in \mathbb{R}^{3}
\end{equation*}%
and the Bessel potential $\left( I-\Delta \right)^{-\gamma/2}$, $\gamma\geq0$, defined by the kernel 
\begin{equation*}
I_{\gamma }\left(x\right) =\left[ \left(4\pi \right)^{\gamma}\Gamma\left(\gamma\right)\right]^{-1}\int_{0}^{\infty}e^{-\pi \left\vert x\right\vert ^{2}/s}e^{-s/4\pi}s^{\left(-3/2+\gamma \right)}\frac{\IntD{s}}{s}, \quad x\in \mathbb{R}^{3}
\end{equation*}%
(see \citep{Stein1970}), is the infinitesimal generator of a strongly continuous bounded holomorphic semigroup of angle $\pi /2$ on $L_{p}\left(\mathbb{R}^{3}\right)$ for $\alpha>0$, $\alpha +\gamma \geq 0$ and any $p\geq 1$, as shown in \citep{AnMc2004}. This semigroup defines the Riesz-Bessel distribution (and the resulting Riesz-Bessel motion) if and only if $\alpha
\in (0,2]$, $\alpha +\gamma \in [0,2]$. When $\gamma =0$, the
fractional Laplacian $-\left( -\Delta \right)^{\alpha /2}$, $\alpha \in (0,2]$, generates the L\'{e}vy $\alpha$-stable distribution. While the exponent of the inverse of the Riesz potential indicates how often large jumps occur, it is the combined effect of the inverses of the Riesz and Bessel potentials that describes the non-Gaussian behaviour of the process. More precisely, depending on the sum $\alpha +\gamma $ of the exponents of the inverses of the Riesz and Bessel potentials, the Riesz-Bessel motion will be either a compound Poisson process, a pure jump process with jumping times dense in $[0,\infty )$ or the sum of a compound Poisson process and an independent Brownian motion. Thus the operator $\mathcal{A}$ is able to generate a range of behaviours of random processes \citep{AnMc2004}.

The equations \eqref{eq:fSPDE} and \eqref{eq:fSCauchy} can be used to describe evolutions of two-stage stochastic systems. The equation \eqref{eq:fSCauchy} determines evolutions on the time interval $[0,t_{0}]$ while \eqref{eq:fSPDE} gives a solution for a system perturbed by fBm on the interval $[t_{0},t_{0}+t]$. CMB is an example of such systems, as it passed through different formation epochs, inflation, recombinatinon etc, see e.g. \citep{Dodelson2003}.

The exact solution of \eqref{eq:fSPDE} is given in the following expansion in terms of spherical harmonics $\shY$, or the \emph{Karhunen-Lo\`{e}ve expansion}:
\begin{align}\label{eq:intro.sol.fSPDE}
  \sol(t) &=  \sum_{\ell=0}^{\infty}\biggl(\sum_{m=-\ell}^{\ell}e^{-\freigv (t+t_{0})}\Fcoe{(\RF_{0})}\shY \notag\\
  &\hspace{1.5cm}+ \sqrt{\vfBm}\Bigl(\int_{0}^{t}e^{-\freigv(t-u)}\IntBa[{\ell 0}](u)\:\shY[\ell,0]\notag\\
          &\hspace{3cm} +\sqrt{2}\sum_{m=1}^{\ell}\bigl(\int_{0}^{t}e^{-\freigv(t-u)}\IntBa(u) \:\CRe \shY \notag\\
          &\hspace{4.6cm}+ \int_{0}^{t}e^{-\freigv(t-u)}\IntBb(u) \:\CIm \shY\bigr)\Bigr)\biggr).
\end{align}
Here, each fractional stochastic integral $\int_{0}^{t}e^{-\freigv(t-s)}\IntB(s)$ is an fBm with mean zero and variance explicitly given, see Section~\ref{sec:sol.fSPDE}, where $(\BMa(u),\BMb(u))$, $m=0,\dots,\ell$, $\ell\in\Nz$, is a sequence of real-valued independent fBms with Hurst index $H$ and variance $1$ (at $t=1$), and $\freigv$ are the eigenvalues of $\psi(-\LBo)$, see Section~\ref{sec:fun.sph2}.

By truncating the expansion \eqref{eq:intro.sol.fSPDE} at degree $\ell=\trdeg$, $\trdeg\ge1$, we obtain an approximation $\trsol(t)$ of the solution $\sol(t)$ of \eqref{eq:fSPDE}. Since the coefficients in the expansion \eqref{eq:intro.sol.fSPDE} can be fast simulated, see e.g. \citep[Section~12.4.2]{KrBo2015}, the approximation $\trsol(t)$ is fully computable and the computation is efficient using the FFT for spherical harmonics $\shY$, see Section~\ref{sec:numer}. We prove that the approximation $\trsol(t)$ of $\sol(t)$, $t>0$ 
(in $L_{2}$ norm on the product space of the probability space $\probSp$ and the sphere $\sph{2}$)
 has the convergence rate $\trdeg^{-\smind}$, $\smind>1$, if the variances $\vfBm$ of the fBm $\fBmsph$ satisfy the smoothness condition $\sum_{\ell=0}^{\infty} \vfBm(1+\ell)^{2\smind+1} <\infty$. This shows that the numerical approximation by truncating the expansion \eqref{eq:intro.sol.fSPDE} is effective and stable.

We also prove that $\sol(t+h)$ has the mean square approximation errors (or the mean quadratic variations) with order $h^{\hurst}$ from $\sol(t)$, as $h\to0+$, for $\hurst\in[1/2,1)$ and $t\ge0$. When $\hurst=1/2$, the Brownian motion case, the convergence rate can be as high as $h$ for $t>0$ (up to a constant). This means that the solution of the fractional SPDE \eqref{eq:fSPDE} evolves continuously with time and the fractional (Hurst) index $\hurst$ affects the smoothness of this evolution.

All above results are verified by numerical examples using an isotropic Gaussian random field as the initial random field. 

CMB is electromagnetic radiation propagating freely through the universe since recombination of ionised atoms and electrons around $370,000$ years after the big bang. As the map of CMB temperature can be modelled as a random field on $\sph{2}$, we apply the truncated solution of the fractional SPDE \eqref{eq:fSPDE} to explore evolutions of the CMB map, using the angular power spectrum of CMB at recombination which was obtained by Planck 2015 results \cite{Planck2016IX} as the initial condition of the Cauchy problem \eqref{eq:fSCauchy}. This gives some indication that the fractional SPDE is flexible enough as a phenomenological model to capture some of the statistical and spectral properties of the CMB that is in equilibrium with an expanding plasma through an extended radiation-dominated epoch.

The paper is organized as follows. Section~\ref{sec:pre} makes necessary preparations. Some results about fractional Brownian motions are derived in Section~\ref{sec:fBm}.  Section~\ref{sec:fSPDE} gives the exact solution of the fractional SPDE \eqref{eq:fSPDE} with fractional Brownian motions and random initial condition from the fractional stochastic Cauchy problem \eqref{eq:fSCauchy}. In Section~\ref{sec:approx.sol}, we give the convergence rate of the approximation errors of truncated solutions in degree and the mean square approximation errors of the exact solution in time. Section~\ref{sec:numer} gives numerical examples.

\section{Preliminaries}\label{sec:pre}
Let $\mathbb{R}^{3}$ be the real
$3$-dimensional Euclidean space with the inner product $\PT{x}\cdot\PT{y}$
for $\PT{x},\PT{y}\in \REuc[3]$ and the Euclidean norm
$|\PT{x}|:=\sqrt{\PT{x}\cdot\PT{x}}$. Let
    $\sph{2}:=\{\PT{x}\in\REuc[3]: |\PT{x}|=1\}$
denote the unit sphere in $\REuc[3]$. The sphere $\sph{2}$ forms a compact
metric space, with the geodesic distance
  $\dist{\PT{x},\PT{y}}:=\arccos(\PT{x}\cdot\PT{y})$ for $\PT{x},\PT{y}\in \sph{2}$ as the metric.

Let $(\probSp,\mathcal{F},\probm)$ be a probability space. Let $\Lpprob[{,\probm}]{2}$ be the $L_{2}$-space on $\probSp$ with respect to the probability measure $\probm$, endowed with the norm $\norm{\cdot}{\Lpprob{2}}$. Let $\rv,\rvb$ be two random variables on $(\probSp,\mathcal{F},\probm)$. Let $\expect{\rv}$ be the expected value of $\rv$, $\cov{\rv,\rvb}:=\expect{(\rv-\expect{\rv})(\rvb-\expect{\rvb})}$ be the covariance between $\rv$ and $\rvb$ and $\var{X}:=\cov{\rv,\rv}$ be the variance of $\rv$.

Let $\Lppsph{2}{2}:=\Lppsph[,\prodpsphm]{2}{2}$ be the real-valued $L_{2}$-space on the product space of $\probSp$ and $\sph{2}$, where $\prodpsphm$ is the corresponding product measure.


\subsection{Functions on $\sph{2}$}\label{sec:fun.sph2}
Let $\Lp{2}{2}=\Lp[,\sphm]{2}{2}$ be a space of all real-valued functions that are square-integrable with respect to the normalized Riemann surface measure $\sphm$ on $\sph{2}$ (that is, $\sphm(\sph{2})=1$), endowed with the $L_{2}$-norm
\begin{equation*}
  \norm{f}{\Lp{2}{2}}:=\left\{\int_{\sph{2}}|f(\PT{x})|^{2}\IntDiff[]{x}\right\}^{1/2}.
\end{equation*}
The space $\Lp{2}{2}$ is a Hilbert space with the inner product
\begin{equation*}
	\InnerL{f,g}:=\InnerL[\Lp{2}{2}]{f,g}:=\int_{\sph{2}}f(\PT{x})g(\PT{x})\IntDiff[]{x},\quad f,g\in\Lp{2}{2}.
\end{equation*}

A \emph{spherical harmonic} of degree $\ell$, $\ell\in \Nz:=\{0,1,2,\dots\}$, on $\sph{2}$ is the restriction to $\sph{2}$ of a homogeneous and harmonic polynomial of total degree $\ell$ defined on $\REuc[3]$.
Let $\shSp[2]{\ell}$ denote the set of all spherical harmonics of exact degree $\ell$ on $\sph{2}$. The dimension of the linear space $\shSp[2]{\ell}$ is $2\ell+1$.
The linear span of
$\shSp[2]{\ell}$, $\ell=0,1,\dots,L$, forms the space
$\sphpo[2]{L}$ of spherical polynomials of degree at most $L$.

Since each pair $\shSp[2]{\ell}$, $\shSp[2]{\ell'}$ for $\ell\neq \ell'\in\Nz$ is $L_{2}$-orthogonal, $\sphpo[2]{L}$ is the direct sum of $\shSp[2]{\ell}$, i.e. $\sphpo[2]{L}=\bigoplus_{\ell=0}^{L} \shSp[2]{\ell}$. The infinite direct sum $\bigoplus_{\ell=0}^{\infty} \shSp[2]{\ell}$ is dense in $\Lp{2}{2}$, see e.g. \citep[Ch.1]{WaLi2006}. For $\PT{x}\in\sph{2}$, using spherical coordinates $\PT{x}:=(\sin\theta \sin\varphi, \sin\theta\cos\varphi,\cos\theta)$, $\theta\in[0,\pi]$, $\varphi\in[0,2\pi)$, the Laplace-Beltrami operator on $\sph{2}$ at $\PT{x}$ is
\begin{equation*}
\LBo:= \frac{1}{\sin\theta}\frac{\partial}{\partial \theta}\left(\sin\theta\frac{\partial}{\partial \theta}\right) + \frac{1}{\sin^{2}\theta}\frac{\partial^{2}}{\partial \varphi^{2}},
\end{equation*}
see \citep[Eq.~1.6.8]{DaXu2013} and also \citep[p.~38]{Muller1966}.
Each member of $\shSp[2]{\ell}$ is an eigenfunction of the negative Laplace-Beltrami operator $-\LBo$ on the sphere $\sph{2}$ with the eigenvalue
\begin{equation}\label{eq:eigenvalue}
  \eigvm:=\ell(\ell+1).
\end{equation}
For $\alpha\ge0$ and $\gamma>0$, using \eqref{eq:psi}, the fractional diffusion operator $\psi(-\LBo)$ in \eqref{eq:fLBo} has
the eigenvalues
\begin{equation}\label{eq:fr.eigvm}
  \freigv =\eigvm^{\alpha/2}(1+\eigvm)^{\gamma/2},\quad \ell\in\Nz,
\end{equation}
see \citep[p.~119--120]{DaLi1990}. By \eqref{eq:eigenvalue} and \eqref{eq:fr.eigvm},
\begin{equation}\label{eq:feigv.est}
  \freigv \asymp (1+\ell)^{\alpha+\gamma},\quad \ell\in\Nz,
\end{equation}
where $a_{\ell}\asymp
b_{\ell}$ means $c\:b_{\ell}\le a_{\ell}\le c' \:b_{\ell}$ for some
positive constants $c$ and $c'$.

A \emph{zonal function} is a function $K:\sph{2}\times\sph{2}\rightarrow \mathbb{R}$ that depends only on the inner product of the arguments, i.e. $K(\PT{x},\PT{y})= \mathfrak{K}(\PT{x}\cdot\PT{y})$,\: $\PT{x},\PT{y}\in \sph{2}$, for some function $\mathfrak{K}:[-1,1]\to \mathbb{R}$.
Let $\Legen{\ell}(t)$, $-1\le t\le1$, $\ell\in\Nz$, be the Legendre polynomial of degree $\ell$.
From \citep[Theorem~7.32.1]{Szego1975}, the zonal function $\Legen{\ell}(\PT{x}\cdot\PT{y})$ is a spherical polynomial of degree $\ell$ of $\PT{x}$ (and also of $\PT{y}$).

Let $\{\shY: \ell\in\Nz,\; m=-\ell,\dots,\ell\}$ be an \emph{orthonormal basis} for the space $\Lp{2}{2}$.
The basis $\shY$ and the Legendre polynomial $\Legen{\ell}(\PT{x}\cdot\PT{y})$ satisfy the \emph{addition theorem}
\begin{equation}\label{eq:addition.theorem}
  \sum_{m=-\ell}^{\ell}Y_{\ell,m}(\PT{x})Y_{\ell,m}(\PT{y})=(2\ell+1)\Legen{\ell}(\PT{x}\cdot \PT{y}).
\end{equation}

In this paper, we focus on the following (complex-valued) orthonormal basis, which are used in physics. Using the spherical coordinates $(\theta,\phi)$ for $\PT{x}$,
\begin{equation}\label{eq:shY.C}
	\shY(\PT{x}) := \sqrt{\frac{(2\ell+1)(\ell-m)!}{(\ell+m)!}}\aLegen{m}(\cos\theta)e^{\imu m\varphi},\quad \ell\in \Nz,\; -\ell\le m\le \ell,
\end{equation}
where $\aLegen{m}(t)$, $t\in[-1,1]$ is the associated Legendre polynomial of degree $\ell$ and order $m$. 

The \emph{Fourier coefficients} for $f$ in $\Lp{2}{2}$ are
\begin{align}\label{eq:Fo coeff}
 \widehat{f}_{\ell m}:=\int_{\sph{2}}f(\PT{x})\shY(\PT{x}) \IntDiff[]{x}, \;\; \ell\in\Nz,\: m=-\ell,\dots,\ell.
\end{align}

Since $\shY=(-1)^{m}\conj{\shY[\ell,-m]}$ and $\widehat{f}_{\ell m}=(-1)^{m}\conj{\widehat{f}_{\ell, -m}}$, for $f\in\Lp{2}{2}$, in $\Lp{2}{2}$ sense,
\begin{equation}\label{eq:f.shY.real.expan}
	f = \sum_{\ell=0}^{\infty}\left(\widehat{f}_{\ell 0}\shY[\ell,0]+2\sum_{m=1}^{\ell}\left(\CRe{\widehat{f}_{\ell m}}\:\CRe{\shY}-\CIm{\widehat{f}_{\ell m}}\:\CIm{\shY}\right)\right).
\end{equation}

Note that the results of this paper can be generalized to any other orthonormal basis.

For $\smind\in\mathbb{R}_{+}$,  
the \emph{generalized Sobolev space} $\sob{2}{\smind}{2}$ is defined as the set of all functions $f\in \Lp{2}{2}$ satisfying
  $\sLB{\smind/2} f \in \Lp{2}{2}$.
The Sobolev space $\sob{2}{\smind}{2}$ forms a Hilbert space with norm
$\norm{f}{\sob{2}{\smind}{2}}:=\normb{\sLB{\smind/2} f}{\Lp{2}{2}}$.
We let $\sob{2}{0}{2}:=\Lp{2}{2}$.

\subsection{Isotropic random fields on $\sph{2}$}
Let $\mathscr{B}(\sph{2})$ denote the Borel $\sigma$-algebra on $\sph{2}$ and let $\RotGr[3]$ be the rotation group on $\REuc[3]$.
\begin{definition}
	An $\mathcal{F}\otimes\mathscr{B}(\sph{2})$-measurable function $T:\probSp\times\sph{2}\to\mathbb{R}$ is said to be a real-valued random field on the sphere $\sph{2}$.
\end{definition}
 We will use $\RF(\PT{x})$ or $\RF(\omega)$ as $\RF(\omega,\PT{x})$ for brevity if no confusion arises.

We say $\RF$ is \emph{strongly isotropic} if for any $k\in\N$ and for all sets of $k$ points $\PT{x}_{1},\cdots,\PT{x}_{k}\in\sph{2}$ and for any rotation $\rho\in \RotGr[3]$, joint distributions of $\RF(\PT{x}_{1}), \dots,\RF(\PT{x}_{k})$ and $\RF(\rho\PT{x}_{1}),\dots,$ $\RF(\rho\PT{x}_{k})$ coinside.

We say $\RF$ is $2$-\emph{weakly isotropic} if for all $\PT{x}\in \sph{2}$ the second moment of $\RF(\PT{x})$ is finite, i.e.
  $\expect{|\RF(\PT{x})|^{2}}<\infty$
and if for all $\PT{x}\in\sph{2}$ and for all pairs of points $\PT{x}_{1},\PT{x}_{2}$ $\in \sph{2}$ and for any rotation $\rho\in\RotGr[3]$ it holds
\begin{equation*}
  \expect{\RF(\PT{x})}=\expect{\RF(\rho\PT{x})},\quad \expect{\RF(\PT{x}_{1})\RF(\PT{x}_{2})}=\expect{\RF(\rho\PT{x}_{1})\RF(\rho\PT{x}_{2})},
\end{equation*}
see e.g. \citep{Adler2009,LaSc2015,MaPe2011}.

In this paper, we assume that a random field $\RF$ on $\sph{2}$ is \emph{centered}, that is, $\expect{\RF(\PT{x})}=0$ for $\PT{x}\in\sph{2}$.

Now, let $\RF$ be $2$-weakly isotropic.
The covariance $\expect{\RF(\PT{x})\RF(\PT{y})}$, because it is rotationally invariant,
is a zonal kernel on $\sph{2}$
\begin{equation*}
  \covarRF(\PT{x}\cdot\PT{y}):=\expect{\RF(\PT{x})\RF(\PT{y})}.
\end{equation*}
This zonal function $\covarRF(\cdot)$ is said to be the \emph{covariance function} for $\RF$.
The covariance function $\covarRF(\cdot)$ is in $\Lpw[{[-1,1]}]{2}$ and has a convergent Fourier expansion
$\covarRF = \sum_{\ell=0}^{\infty}\APS (2\ell+1) \Legen{\ell}$ in $\Lpw[{[-1,1]}]{2}$.
The set of Fourier coefficients
\begin{equation*}
    \APS
    := \int_{\sph{2}}\covarRF(\PT{x}\cdot\PT{y}) \Legen{\ell}(\PT{x}\cdot\PT{y}) \IntDiff[]{x}
    =\frac{1}{2\pi} \int_{-1}^{1} \covarRF(t)\Legen{\ell}(t) \IntD{t}
\end{equation*}
is said to be the \emph{angular power spectrum} for the random field $\RF$, where the second equality follows by the properties of zonal functions.

By the addition theorem in \eqref{eq:addition.theorem} we can write
\begin{equation}\label{eq:covariance.G.expan}
  \expect{\RF(\PT{x})\RF(\PT{y})}=\covarRF(\PT{x}\cdot\PT{y})= \sum_{\ell=0}^{\infty} \APS (2\ell+1) \Legen{\ell}(\PT{x}\cdot\PT{y})=\sum_{\ell=0}^{\infty} \APS \sum_{m=-\ell}^{\ell} \shY(\PT{x})\shY(\PT{y}).
\end{equation}

We define \emph{Fourier coefficients} for a random field $\RF$ by, cf. \eqref{eq:Fo coeff},
\begin{equation}\label{eq:Fcoe.RF}
  \widehat{\RF}_{\ell m} := \InnerL{\RF, \shY},\quad \ell\in\Nz,\; m=-\ell,\dots, \ell.
\end{equation}

The following lemma, from \citep[p.~125]{MaPe2011} and \citep[Lemma~4.1]{LeSlWaWo2017}, shows the orthogonality of the Fourier coefficients $\Fcoe{\RF}$ of $\RF$.
\begin{lemma}[\citep{LeSlWaWo2017,MaPe2011}]\label{lem:orth.Fcoe.RF} Let $\RF$ be a $2$-weakly isotropic random field on $\sph{2}$ with angular power spectrum $\APS$. Then for $\ell,\ell'\ge0$, $m=-\ell,\dots,\ell$ and $m'=-\ell',\dots,\ell'$,
\begin{equation}\label{eq:orth.Fcoe.RF}
  \expect{\Fcoe{\RF}\Fcoe[\ell' m']{\RF}}=\APS \Kron \Kron[m m'],
\end{equation}
where $\Kron$ is the Kronecker delta.
\end{lemma}

We say $\RF$ a \emph{Gaussian random field} on $\sph{2}$ if for each $k\in\N$ and $\PT{x}_{1}, \dots,\PT{x}_{k}\in \sph{2}$, the vector $(\RF(\PT{x}_1),\dots,\RF(\PT{x}_{k}))$ has a multivariate Gaussian distribution.

We note that a Gaussian random field is strongly isotropic if and only if it is $2$-weakly isotropic, see e.g. \citep[Proposition~5.10(3)]{MaPe2011}.

\section{Fractional Brownian motion}\label{sec:fBm}
Let $\hurst\in(1/2,1)$ and $\sigma>0$. A \emph{fractional Brownian motion} (fBm) $\fBm(t), t\ge0$, with index $\hurst$ and variance $\sigma^{2}$ at $t=1$ is a centered Gaussian process on $\Rplus$ satisfying
\begin{equation*}
	\fBm(0)=0,\quad \expect{\bigl|\fBm(t)-\fBm(s)\bigr|^{2}} = |t-s|^{2\hurst}\sigma^{2}.
\end{equation*}
The constant $\hurst$ is called the \emph{Hurst index}. See e.g. \citep{BiHuOkZh2008}.
By the above definition, the variance of $\fBm(t)$ is $\expect{|\fBm(t)|^{2}}=t^{2\hurst}\sigma^{2}$. 

For convenience, we use $\fBm[1/2](t)$ (with $\sigma=1$) to denote the Brownian motion (or the Wiener process) on $\Rplus$.

\begin{definition}\label{defn:complex.fBm}
	Let $\hurst\in[1/2,1)$. Let $\beta^{1}(t)$ and $\beta^{2}(t)$ be independent real-valued fBms with the Hurst index $\hurst$ and variance $1$ (at $t=1$). A complex-valued fractional Brownian motion $\fBm(t)$, $t\ge0$, with Hurst index $\hurst$ and variance $\sigma^{2}$ can be defined as
\begin{equation*}
	\fBm(t) = \bigl(\beta^{1}(t) + \imu \beta^{2}(t)\bigr)\sigma.
\end{equation*}
\end{definition}

We define the \emph{$\Lp{2}{2}$-valued fractional Brownian motion} $\fBmsph(t)$ as follows, see Grecksch and Anh \citep[Definition~2.1]{GrAn1999}.

\begin{definition}\label{defn:fBmsph}
Let $\hurst\in[1/2,1)$. Let $\vfBm>0$, $\ell\in\Nz$ satisfying  $\sum_{\ell=0}^{\infty}(2\ell+1)\vfBm<\infty$. Let $\fBm_{\ell m}(t)$, $t\ge0$, $\ell\in\Nz, m=-\ell,\dots,\ell$ be a sequence of independent complex-valued fractional Brownian motions on $\Rplus$ with Hurst index $\hurst$, and variance $\vfBm$ at $t=1$ and $\CIm{\fBm_{\ell 0}(t)}=0$ for $\ell\in\Nz$, $t\ge0$.
For $t\ge0$, the $\Lp{2}{2}$-valued fractional Brownian motion is defined by the following expansion (in $\Lppsph{2}{2}$ sense) in spherical harmonics with fBms $\fBm_{\ell m}(t)$ as coefficients:
\begin{equation}\label{eq:fBmsph}
	\fBmsph(t,\PT{x}):=\sum_{\ell=0}^{\infty}\sum_{m=-\ell}^{\ell}\fBm_{\ell m}(t) \shY(\PT{x}),\quad \PT{x}\in\sph{2}.
\end{equation}
\end{definition}
We also call $\fBmsph(t,\PT{x})$ in Definition~\ref{defn:fBmsph} a fractional Brownian motion on $\sph{2}$. 

The fBm $\fBmsph(t,\PT{x})$ in \eqref{eq:fBmsph} is well-defined since for $t\ge0$, by Parseval's identity,
\begin{equation*}
	\expect{\norm{\fBmsph(t,\cdot)}{\Lp{2}{2}}^{2}}\le \sum_{\ell=0}^{\infty}\sum_{m=-\ell}^{\ell}\expect{\left|\fBm_{\ell m}(t) \right|^{2}}=t^{2\hurst}\sum_{\ell=0}^{\infty}(2\ell+1)\vfBm<\infty.
\end{equation*}

We let in this paper $\fBmsph(t,\PT{x})$ be real-valued. For $\ell\in\Nz$, let 
\begin{align*}
	\sqrt{\vfBm}\:\BMa[\ell0](t)&:=\fBm_{\ell 0}(t), \;\; \BMb[\ell0](t):=\BMa[\ell0](t),\\
	\sqrt{\frac{\vfBm}{2}}\BMa(t)&:=\CRe{\fBm_{\ell m}(t)},\;\;
	\sqrt{\frac{\vfBm}{2}}\BMb(t):=-\CIm{\fBm_{\ell m}(t)}=\CIm{\fBm_{\ell m}(t)},\quad m=1,\dots,\ell,
\end{align*}
in law.
Then, $(\BMa,\BMb)$, $m=0,\dots,\ell$, $\ell\in\Nz$, is a sequence of independent fBms with Hurst index $H$ and variance $1$ (at $t=1$). 

By \eqref{eq:f.shY.real.expan}, we can write \eqref{defn:complex.fBm} as, for $t\ge0$, in $\Lppsph{2}{2}$ sense, 
\begin{equation}\label{eq:fBmsph.real.expan}
	\fBmsph(t) = \sum_{\ell=0}^{\infty}\sqrt{\vfBm}\left(\BMa[\ell 0](t)\shY[\ell,0]+\sqrt{2}\sum_{m=1}^{\ell}\left(\BMa(t)\:\CRe{\shY} 
	+\BMb(t)\:\CIm{\shY}\right)\right).
\end{equation}

\begin{remark}
For a Hilbert space $H$, Duncan et al. \cite{DuJaPa2006,DuPaMa2002} introduced an $H$-valued cylindrical fractional Brownian motion, whereas our approach yields $Q$-fractional Brownian motions for a kernel operator $Q$. But Duncan et al.'s results can be modified for $Q$-fractional Brownian motions.
\end{remark}

\begin{remark}
We give the covariance of $\fBmsph(t,\PT{x})$, as follows. As in Definition~\ref{defn:fBmsph}, $\hurst\in[1/2,1)$, $\vfBm>0$, $\sum_{\ell=0}^{\infty}(2\ell+1)\vfBm <\infty$, $\{\fBm_{\ell m}(t)| \ell\in\Nz, m=-\ell,\dots,\ell\}$ is a sequence of centered independent fBms with Hurst index $\hurst$. Note that 
\begin{align*}
	&\expect{|\fBm_{\ell m}(1)|^{2}} = \vfBm, \quad \expect{\fBm_{\ell m}(t)} = 0,\quad \mbox{for~} t\ge0 \mbox{~and all~} \ell\in\Nz, m=-\ell,\dots,\ell,\\
	&\sum_{\ell=0}^{\infty}\sum_{m=-\ell}^{\ell}\expect{|\fBm_{\ell m}(t)|^{2}} = t^{2\hurst}\sum_{\ell=0}^{\infty}(2\ell+1)\vfBm.
\end{align*}
Then, for $t\ge0$, $\PT{x},\PT{y}\in\sph{2}$,
\begin{align}\label{eq:cov.fBm}
	\expect{\fBmsph(t,\PT{x})\fBmsph(t,\PT{y})}
	&= \expect{\sum_{\ell=0}^{\infty}\sum_{m=-\ell}^{\ell}\fBm_{\ell m}(t)\shY(\PT{x})\sum_{\ell'=0}^{\infty}\sum_{m'=-\ell'}^{\ell'}\fBm_{\ell' m'}(t)\shY[\ell',m'](\PT{y})}\notag\\
	&= \sum_{\ell=0}^{\infty}\sum_{\ell'=0}^{\infty}\sum_{m=-\ell}^{\ell}\sum_{m'=-\ell'}^{\ell'}\expect{\fBm_{\ell m}(t)\fBm_{\ell' m'}(t)}\shY(\PT{x})\shY[\ell',m'](\PT{y})\notag\\
	&= \sum_{\ell=0}^{\infty}\sum_{m=-\ell}^{\ell}
	\expect{|\fBm_{\ell m}(t)|^{2}}\shY(\PT{x})\shY(\PT{y})\notag\\
	&= t^{2\hurst}\sum_{\ell=0}^{\infty}(2\ell+1)\vfBm \Legen{\ell}(\PT{x}\cdot\PT{y}),
\end{align}
where the last equality uses \eqref{eq:addition.theorem}.
\end{remark}

For a bounded measurable function $g$ on $\Rplus$ (which is deterministic), the stochastic integral $\int_{s}^{t}g(u)\IntD{\fBm_{\ell m}(u)}$ can be defined as a Riemann-Stieltjes integral, see \citep{Lin1995}.
The \emph{$\Lp{2}{2}$-valued stochastic integral} $\int_{s}^{t}g(u)\IntD{\fBmsph(u)}$ can then be defined as an expansion in spherical harmonics with coefficients $\int_{s}^{t}g(u)\IntD{\fBm_{\ell m}(u)}$, as follows. 
\begin{definition}\label{defn:IntD.fBmsph}
	Let $\hurst\in[1/2,1)$ and let $\fBmsph(t)$ be an $\Lp{2}{2}$-valued fBm with the Hurst index $\hurst$. For $t>s\ge0$, the fractional stochastic integral $\int_{s}^{t}g(u)\IntD{\fBmsph(u)}$ for a bounded measurable function $g$ on $\Rplus$ is defined by, in $\Lppsph{2}{2}$ sense,
	\begin{equation*}
		\int_{s}^{t}g(u)\IntD{\fBmsph(u)} := \sum_{\ell=0}^{\infty}\sum_{m=-\ell}^{\ell}\left(\int_{s}^{t}g(u)\IntD{\fBm_{\ell m}(u)}\right) \shY.
	\end{equation*}
\end{definition}

The following theorem \citep[Lemma~2.3]{GrAn1999} provides an upper bound for $\int_{s}^{t}g(u)\IntD{\fBmsph(u)}$.
\begin{proposition}\label{thm:BD.Int.fBmsph}
	Let $\hurst\in[1/2,1)$. Let $g$ be a bounded measurable function on $\Rplus$, and if $\hurst>1/2$,
	$\int_{s}^{t}\int_{s}^{t}g(u)g(v)|u-v|^{2\hurst-2}\IntD{u}\IntD{v}<\infty$.
	For $t>s\ge0$, the fractional stochastic integral $\int_{s}^{t}g(u)\IntD{\fBmsph(u)}$ given by Definition~\ref{defn:IntD.fBmsph} satisfies
	\begin{equation*}
		\expect{\norm{\int_{s}^{t}g(u)\IntD{\fBmsph(u)}}{\Lp{2}{2}}^{2}} \le C\: \left(\int_{s}^{t}|g(u)|^{\frac{1}{\hurst}}\IntD{u}\right)^{2\hurst} \sum_{\ell=0}^{\infty}(2\ell+1)\vfBm,
	\end{equation*}
	where the constant $C$ depends only on $\hurst$.
\end{proposition}

\section{Fractional SPDE}\label{sec:fSPDE}
This section studies the Karhunen-Lo\`{e}ve expansion of the solution to the fractional SPDE \eqref{eq:fSPDE} with the fractional diffusion (Laplace-Beltrami) operator $\psi(-\LBo)$ in \eqref{eq:fLBo} and the $\Lp{2}{2}$-valued fractional Brownian motion $\fBmsph(t)$ given in Definition~\ref{defn:fBmsph}. The random initial condition is a solution of the fractional stochastic Cauchy problem \eqref{eq:fSCauchy}. We will give the convergence rates for the approximation errors of the truncated Karhunen-Lo\`{e}ve expansion in degree and the mean square approximation errors in time of the solution of \eqref{eq:fSPDE}.


\subsection{Random initial condition as a solution of fractional stochastic Cauchy problem}\label{sec:fSCauchy}
Let
\begin{equation}\label{eq:ran.cond}
	\RF_{0}(\PT{x}) := \sum_{\ell=0}^{\infty} \sum_{m=-\ell}^{\ell}\Fcoe{(\RF_{0})} \shY(\PT{x}),\quad \PT{x}\in\sph{2}
\end{equation}
be a centered, $2$-weakly isotropic Gaussian random field on $\sph{2}$. Let the sequence $\{\vT\}_{\ell\in\Nz}$ be the angular power spectrum of $\RF_{0}$. Let
\begin{equation}\label{eq:vTc}
	\vTc := \left\{\begin{array}{ll}
		\vT[0], & \ell=0,\\
		\vT/2, & \ell\ge1.
	\end{array}
	\right.
\end{equation}
Then, $\Fcoe{(\RF_{0})}$ follows the normal distribution $\nrml(0,\vTc)$.

Let $\alpha\ge0,\gamma>0$. By \citep[Theorem~3]{DOvidio2014}, the solution to the fractional stochastic Cauchy problem \eqref{eq:fSCauchy}
is
\begin{equation}\label{eq:sol.fr.stoch.Cauchy}
  \solC(t,\PT{x})= \sum_{\ell=0}^{\infty}\sum_{m=-\ell}^{\ell} e^{-\freigv t}\Fcoe{(\RF_{0})}\shY(\PT{x}).
\end{equation}

\begin{remark}
A time-fractional extension of Cauchy problem \eqref{eq:fSCauchy} and its solution \eqref{eq:sol.fr.scalar.SPDE} was studied by D'Ovidio, Leonenko and Orsingher \cite{DOLeOr2016}. In this paper, we do not consider the initial condition from a time-fractional Cauchy problem.	
\end{remark}

For $t\ge0$, $\solC(t,\PT{x})$ is a $2$-weakly isotropic Gaussian random field, as shown by the following theorem.
\begin{proposition}\label{prop:isotr.sol.Cauchy}
	Let $\solC(t,\PT{x})$ be the solution in \eqref{eq:sol.fr.stoch.Cauchy} of the fractional stochastic Cauchy problem \eqref{eq:fSCauchy}. Then, for any $t\ge 0$, $\solC(t,\PT{x})$ is a $2$-weakly isotropic Gaussian random field on $\sph{2}$, and its Fourier coefficients satisfy for $\ell,\ell'\in \Nz, m=-\ell,\dots,\ell$ and $m'=-\ell',\dots,\ell'$,
	\begin{equation}\label{eq:Fcoe.sol.Cauchy.cov}
		\expect{\Fcoe{\solC(t)}\Fcoe[\ell' m']{\solC(t)}} = e^{-2\freigv t} \vTc\: \Kron\Kron[m m'],
	\end{equation}
	where we let $\Fcoe{\solC(t)}:=\Fcoe{(\solC(t))}$ for simplicity and $\vTc$ is given by \eqref{eq:vTc}.
\end{proposition}

\subsection{Solution of fractional SPDE}\label{sec:sol.fSPDE}
The following theorem gives the exact solution of the fractional SPDE in \eqref{eq:fSPDE}
under the random initial condition $\sol(0)=\solC(t_{0})$ for $t_{0}\ge0$.

\begin{theorem}\label{thm:sol.fr.SPDE.scalar}
    Let $\hurst\in[1/2,1)$, $\alpha\ge0$, $\gamma>0$. Let $\fBmsph(t)$ be a fractional Brownian motion on the sphere $\sph{2}$ with Hurst index $\hurst$ and variances $\vfBm$. Let $\solC(t,\PT{x})$ be the solution in \eqref{eq:sol.fr.stoch.Cauchy} of the fractional stochastic Cauchy problem \eqref{eq:fSCauchy}. Then, for $t_{0}\ge0$, the solution to the equation \eqref{eq:fSPDE} under the random initial condition $\sol(0) = \solC(t_{0})$ is for $t\ge0$,
\begin{align}\label{eq:sol.fr.scalar.SPDE}
  \sol(t) &=  \sum_{\ell=0}^{\infty}\biggl(\sum_{m=-\ell}^{\ell}e^{-\freigv (t+t_{0})}\Fcoe{(\RF_{0})}\shY \notag\\
  &\hspace{1.5cm}+ \sqrt{\vfBm}\Bigl(\int_{0}^{t}e^{-\freigv(t-u)}\IntBa[{\ell 0}](u)\:\shY[\ell,0]\notag\\
          &\hspace{3cm} +\sqrt{2}\sum_{m=1}^{\ell}\bigl(\int_{0}^{t}e^{-\freigv(t-u)}\IntBa(u) \:\CRe \shY \notag\\
          &\hspace{4.6cm}+ \int_{0}^{t}e^{-\freigv(t-u)}\IntBb(u) \:\CIm \shY\bigr)\Bigr)\biggr).
\end{align}
\end{theorem}

Let 
\begin{equation*}
	\igamma{a,z}:=\frac{1}{\Gamma(a)}\int_{0}^{1}t^{a-1}e^{-zt}\IntD{t}, \quad \CRe{a}>0,\; z\in\mathbb{C}
\end{equation*}
 be the incomplete gamma function, see e.g. \citep[Eq.~8.2.7]{NIST:DLMF}.

For $\ell\in\Nz$, $t\ge0$ and $\hurst\in[1/2,1)$, let
\begin{equation}\label{eq:var.int.fBm}
	\bigl(\GSD^{\hurst}\bigr)^{2} := \hurst \Gamma(2\hurst) t^{2\hurst} \left(e^{-2\freigv t}\:\igamma{2\hurst,-\freigv t} + \igamma{2\hurst,\freigv t}\right).
\end{equation}
We write $\GSD:=\GSD^{\hurst}$ and $\Gvar:=\bigl(\GSD^{\hurst}\bigr)^{2}$ if no confusion arises.

For $\hurst=1/2$, the formula \eqref{eq:var.int.fBm} reduces to
	\begin{equation}\label{eq:var.int.BM}
		\bigl(\GSD^{1/2}\bigr)^{2}=\left\{
    \begin{array}{ll}
    t, & \ell=0,\\[1mm]
    \displaystyle\frac{1-e^{-2\freigv t}}{2\freigv}, &\ell\ge1.
    \end{array}\right.
	\end{equation}

For $t\ge0$, the fractional stochastic integrals 
\begin{equation}\label{eq:fSI.0t}
	\int_{0}^{t}e^{-\freigv(t-u)}\IntB(u),\quad i=1,2,
\end{equation}
 in the expansion \eqref{eq:sol.fr.scalar.SPDE} in Theorem~\ref{thm:sol.fr.SPDE.scalar} are normally distributed with means zero and variances $\Gvar$, as a consequence of the following proposition.

\begin{proposition}\label{prop:Xt.coeff.fSI} Let $\hurst\in[1/2,1)$, $\alpha\ge0,\gamma>0$, and $\freigv$ be given by \eqref{eq:fr.eigvm}. Let $t>s\ge0$. For $m=0,\dots,\ell$, $\ell\in\Nz$ and $i=1,2$, each fractional stochastic integral
	\begin{equation}\label{eq:fBm.int}
	 \int_{s}^{t}e^{-\freigv(t-u)}\IntB(u)
	 \end{equation}
	is normally distributed with mean zero and variance $\Gvar[\ell,t-s]$ as given by \eqref{eq:var.int.fBm}.
	
	Moreover, for $1/2<\hurst<1$,
	\begin{align}\label{eq:Gvar.UB}
		\Gvar[\ell,t-s]
		&=\expect{\left|\int_{s}^{t}e^{-\freigv(t-u)}\IntB(u)\right|^{2}}\notag\\
		&=\norm{\int_{s}^{t}e^{-\freigv(t-u)}\IntB(u)}{\Lpprob{2}}^{2}\notag\\
		&\le C (t-s)^{2\hurst},
	\end{align}
	where the constant $C$ depends only on $\hurst$.
\end{proposition}

\begin{remark}
By the expansion of solution $\sol(t,\PT{x})$ in \eqref{eq:sol.fr.scalar.SPDE} of \eqref{eq:fSPDE} under the initial condition $\sol(0,\PT{x})=\RF_{0}(\PT{x})$ and the distribution coefficients of the expansion given by Proposition~\ref{prop:Xt.coeff.fSI}, the covariance of $\sol(t,\PT{x})$ can be obtained using the techniques of Proposition~\ref{prop:Gvar.diff.t.fBm} and \eqref{eq:cov.fBm}.

The solution of a Cauchy problem defined by the Riesz-Bessel operator yields an $\alpha$-stable type solution, which is non-Gaussian. But, when we bring the fBm noise into the model, Theorem~\ref{thm:trsol.err} shows that convergence is dominated by fBm, and under the condition of this theorem, the fBm driving the equation has a continuous version, hence the solution $\sol(t)$ swings, but would not jump. It is not safe to assume that the solution is non-Gaussian (jumpy).
\end{remark}

The following proposition shows the change rate of the variance of fractional stochastic integral \eqref{eq:fSI.0t} with respect to time.

\begin{proposition}\label{prop:Gvar.diff.t.fBm}
	Let $\hurst\in[1/2,1)$, $\alpha\ge0,\gamma>0$, and $\freigv$ be given by \eqref{eq:fr.eigvm}. For $t\ge0$, $m=0,\dots,\ell$, $\ell\in\Nz$ and $i=1,2$, the variance $\Gvar[\ell,t]$ of the fractional stochastic integral
	\eqref{eq:fSI.0t}
	satisfies as $h\to0+$:\\
(i) for $\hurst=1/2$,
when $t=0$,
	\begin{equation*}
		\left|\GSD[\ell,t+h]-\GSD[\ell,t]\right|\le h^{1/2},
	\end{equation*}
when $t>0$,
	\begin{equation*}
		\left|\GSD[\ell,t+h]-\GSD[\ell,t]\right|\le C_{1} h,
	\end{equation*}
where the constant 
\begin{equation*}
C_{1}:=\left\{
	\begin{array}{ll}
		\displaystyle\frac{1}{2\sqrt{t}}, & \ell=0,\\[5mm]
		\displaystyle\sqrt{\frac{\freigv}{2(1-e^{-2\freigv t})}}\:e^{-2\freigv t}, & \ell\ge1;
	\end{array}
\right.
\end{equation*}
(ii) for $\hurst\in (1/2,1)$,
	\begin{equation*}
		\left|\GSD[\ell,t+h]-\GSD[\ell,t]\right|\le C_{2} \left(1+ \freigv t^{\hurst} h^{1-\hurst}\right)h^{\hurst},
	\end{equation*}
where the constant $C_{2}$ depends only on $\hurst$.
\end{proposition}

Proposition~\ref{prop:Gvar.diff.t.fBm} implies the following common upper bound for all $\hurst\in[1/2,1)$.
\begin{corollary}\label{corol:Gvar.diff.t.fBm}
	Let $\hurst\in[1/2,1)$, $\alpha\ge0,\gamma>0$, and $\freigv$ be given by \eqref{eq:fr.eigvm}. For $t\ge0$, $m=0,\dots,\ell$, $\ell\in\Nz$ and $i=1,2$, the variance $\Gvar[\ell,t]$ of the fractional stochastic integral \eqref{eq:fSI.0t} satisfies as $h\to0+$,
	\begin{equation*}
		\left|\GSD[\ell,t+h]-\GSD[\ell,t]\right|\le C (1+\freigv) h^{\hurst},
	\end{equation*}
where the constant $C$ depends only on $\hurst$, $\alpha$, $\gamma$ and $t$.
\end{corollary}

\subsection{Approximation to the solution}\label{sec:approx.sol}
In this section, we truncate the Karhunen-Lo\`{e}ve expansion of the solution $\sol(t)$ in \eqref{eq:sol.fr.scalar.SPDE} of the fractional SPDE \eqref{eq:fSPDE} for computational implementation. We give an estimate for the approximation error of the truncated expansion. We also derive an upper bound for the mean square approximation errors in time for the solution $\sol(t)$. 

\subsubsection{Truncation approximation to Karhunen-Lo\`{e}ve expansion}
\begin{definition}\label{defn:KL.approx}
For $t\ge0$ and $\trdeg\in\Nz$, the Karhunen-Lo\`{e}ve approximation $\trsol(t)$ of (truncation) degree $\trdeg$ to the solution $\sol(t)$ is
\begin{align}\label{eq:trsol.fr.scalar.SPDE}
  \trsol(t) &= \sum_{\ell=0}^{\trdeg}\biggl(\sum_{m=-\ell}^{\ell}e^{-\freigv t}\InnerL{\solC(t_{0}),\shY}\shY\notag\\
  &\hspace{1.5cm} + \sqrt{\vfBm}\Bigl(\int_{0}^{t}e^{-\freigv(t-u)}\IntBa[{\ell 0}](u)\:\shY[\ell,0]\notag\\
          &\hspace{3cm} +\sqrt{2}\sum_{m=1}^{\ell}\bigl(\int_{0}^{t}e^{-\freigv(t-u)}\IntBa(u) \:\CRe \shY\notag\\
          &\hspace{4.6cm} + \int_{0}^{t}e^{-\freigv(t-u)}\IntBb(u) \:\CIm \shY\bigr)\Bigr)\biggr).
\end{align}
\end{definition}

	For $\ell\in\Nz$, let
	\begin{subequations}\label{eq:trsol}
	\begin{align}
		\trsola(t) &:=\sum_{m=-\ell}^{\ell}e^{-\freigv t}\InnerL{\solC(t_{0}),\shY}\shY,\label{eq:trsol.a}\\
		\trsolb(t) &:=\sqrt{\vfBm}\Bigl(\int_{0}^{t}e^{-\freigv(t-u)}\IntBa[{\ell 0}](u)\:\shY[\ell,0]\notag\\
          &\hspace{1.7cm} +\sqrt{2}\sum_{m=1}^{\ell}\bigl(\int_{0}^{t}e^{-\freigv(t-u)}\IntBa(u) \:\CRe \shY \notag\\
          &\hspace{3.6cm}+ \int_{0}^{t}e^{-\freigv(t-u)}\IntBb(u) \:\CIm \shY\bigr)\Bigr)\label{eq:trsol.b}.
	\end{align}	
	\end{subequations}
	By \citep[Remark~6.13]{MaPe2011}, $\trsola(t)$ and $\trsolb(t)$, $t\ge0$, $\ell\in\Nz$, are centered Gaussian random fields. 
	
	The following theorem shows that the convergence rate of the Karhunen-Lo\`{e}ve approximation in \eqref{eq:trsol.fr.scalar.SPDE} of the exact solution in \eqref{eq:sol.fr.scalar.SPDE} is determined by the convergence rate of variances $\vfBm$ of the fBm $\fBmsph(t)$ (with respect to $\ell$).
	
\begin{theorem}\label{thm:trsol.err}
 Let $\sol(t)$ be the solution \eqref{eq:sol.fr.scalar.SPDE} to the fractional SPDE in \eqref{eq:fSPDE} with $\sol(0)=\solC(t_{0})$, $t_{0}\ge0$, and the fBm $\fBmsph(t)$ whose variances $\vfBm$ satisfy $\sum_{\ell=0}^{\infty} \vfBm(1+\ell)^{2\smind+1} <\infty$ for $r>1$. Let $L\ge1$ and let $\trsol(t)$ be the Karhunen-Lo\`{e}ve approximation of $\sol(t)$ given in \eqref{eq:trsol.fr.scalar.SPDE}. For $t>0$, the truncation error
	\begin{equation}\label{eq:sol.tr.err}
		\norm{\sol(t)-\trsol(t)}{\Lppsph{2}{2}} \le C \trdeg^{-\smind},
	\end{equation}
	where the constant $C$ depends only on $\alpha$, $\gamma$, $t_{0}$, $t$ and $\smind$.
\end{theorem}
\begin{remark} Given $t_{1}>0$, the truncation error in \eqref{eq:sol.tr.err} is uniformly bounded on $[t_{1},+\infty)$:
	\begin{equation*}
		\sup_{t\ge t_{1}}\norm{\sol(t)-\trsol(t)}{\Lppsph{2}{2}} \le C \trdeg^{-\smind},
	\end{equation*}
	where the constant $C$ depends only on $\alpha$, $\gamma$, $t_{0}$, $t_{1}$ and $\smind$.
	
	The condition $r>1$ comes from the Sobolev embedding theorem (into the space of continuous functions) on $\sph{2}$, see \citep{Kamzolov1982}. This implies that the random field $\fBmsph(t,\cdot)$, $t\ge0$, has a representation by a continuous function on $\sph{2}$ almost surely, which allows numerical computations to proceed, see \citep{AnLa2014,LaSc2015}.
\end{remark}

\subsubsection{Mean square approximation errors in time}
For $t\ge0$, let $\{(\incsol[\ell m]^{1}(t),\incsol[\ell m]^{2}(t))| m=0,\dots,\ell, \ell\in\Nz\}$ be a sequence of independent and standard normally distributed random variables. Let
\begin{equation}\label{eq:U.l}
	\incsol(t):= \incsola[\ell 0](t) \:\shY[\ell 0] + \sqrt{2} \sum_{m=1}^{\ell}\bigl(\incsola(t)\:\CRe\shY
	+ \incsolb(t)\:\CIm\shY\bigr).
\end{equation}
This is a Gaussian random field and the series $\sum_{\ell=0}^{\infty}\incsol(t)$ converges to a Gaussian random field on $\sph{2}$ (in $\Lppsph{2}{2}$ sense), see \citep[Remark~6.13 and Theorem~5.13]{MaPe2011}. Let $\incsol[](t):=\sum_{\ell=0}^{\infty}\incsol(t)$. By Lemma~\ref{lem:orth.Fcoe.RF},
\begin{equation}\label{eq:Fcoe.incsol.cov}
	\expect{\Fcoe{\incsol[](t)}\Fcoe[\ell' m']{\incsol[](t)}} = \Kron\Kron[m m'],
\end{equation}
where we let $\Fcoe{\incsol[](t)}:=\Fcoe{(\incsol[](t))}$ for brevity.

Let $\trsola(t)$ and $\trsolb(t)$ be given in \eqref{eq:trsol.a} and \eqref{eq:trsol.b} respectively. We let
\begin{equation}\label{eq:trsol.ell}
	\trsol[\ell](t):=\trsola(t)+\trsolb(t).
\end{equation}

For $t\ge0$ and $h>0$, the following theorem shows that $\trsol[\ell](t+h)$ can be represented by $\trsol[\ell](t)$ and $\incsol(t)$.
\begin{lemma}\label{lem:sol.represent}
	Let $\hurst\in[1/2,1)$ and $\ell\in\Nz$. Let $\trsol[\ell](t)$ be given by \eqref{eq:trsol.ell}. Then, for $t\ge0$ and $h>0$,
	\begin{equation}
		\trsol[\ell](t+h) = e^{-\freigv h} \trsol[\ell](t) + \sqrt{\vfBm}\:\GSD[\ell,h] \incsol(t),\label{eq:variat.sol.time}		
	\end{equation}
	where $\GSD[\ell,h]:=\GSD[\ell,h]^{\hurst}$ is given by \eqref{eq:var.int.fBm} and $\incsol(t)$ is given by \eqref{eq:U.l}.
\end{lemma}
\begin{remark} In particular, the equation \eqref{eq:variat.sol.time} implies for $\ell\in\Nz$ and $t>0$,
\begin{equation}\label{eq:variat.sol.0}
  \trsol[\ell](t) = e^{-\freigv t} \trsol[\ell](0) + \sqrt{\vfBm}\:\GSD \incsol(0).
\end{equation}
\end{remark}

The following theorem gives an estimate for the mean square approximation errors for $\sol(t)$ in time, which depends on the Hurst index $\hurst$ of the fBm $\fBmsph(t)$.
\begin{theorem}\label{thm:UB.fBm.L2err.sol}
	Let $\sol(t)$ be the solution in \eqref{eq:sol.fr.scalar.SPDE} to the equation \eqref{eq:fSPDE}, where the angular power spectrum $\vT$ for the initial random field $\RF_{0}$ and the variances $\vfBm$ for the fBm $\fBmsph(t)$ satisfy $\sum_{\ell=0}^{\infty}(2\ell+1)(\vT+(1+\freigv)^{2} \vfBm)<\infty$. Then, for $t\ge0$, as $h\to0+$,
	\begin{equation}
		\norm{\sol(t+h) - \sol(t)}{\Lppsph{2}{2}} \le C h^{\hurst},\label{eq:UB.fBm.L2err.sol}
	\end{equation}
	where the constant $C$ depends only on $\alpha$, $\gamma$, $t$, $\vT$ and $\vfBm$.
\end{theorem}
\begin{remark}
	Given $t_1>0$, as $h\to0+$, the truncation error in \eqref{eq:UB.fBm.L2err.sol} is uniformly bounded on $[t_1,+\infty)$:
	\begin{equation*}
		\sup_{t\ge t_{1}}\norm{\sol(t+h) - \sol(t)}{\Lppsph{2}{2}} \le C h^{\hurst},
	\end{equation*}
	where the constant $C$ depends only on $\alpha$, $\gamma$, $t_{1}$, $\vT$ and $\vfBm$.
\end{remark}

The mean square approximation errors of the truncated solutions $\trsol(t)$ have the same convergence rate $h^{\hurst}$ as $\sol(t)$, as we state below. The proof is similar to that of Theorem~\ref{thm:UB.fBm.L2err.sol}.
\begin{corollary} Under the conditions of Theorem~\ref{thm:UB.fBm.L2err.sol}, for $t_{1}>0$, as $h\to0+$:
	\begin{equation*}
		\sup_{t\ge t_{1}, \trdeg\in\Nz}\norm{\trsol(t+h) - \trsol(t)}{\Lppsph{2}{2}} \le C h^{\hurst},
	\end{equation*}
	where the constant $C$ depends only on $\alpha$, $\gamma$, $t_{1}$, $\vT$ and $\vfBm$.
\end{corollary}

For $\hurst=1/2$ and $t>0$, the convergence order of the upper bound in \eqref{eq:UB.fBm.L2err.sol} can be improved to $h$, as we state in the following corollary. The proof is similar to that of Theorem~\ref{thm:UB.fBm.L2err.sol} but needs to use Proposition~\ref{prop:Gvar.diff.t.fBm}~(i). 
\begin{corollary}\label{cor:MQV.Bm}
	Let $\sol(t)$ be the solution in \eqref{eq:sol.fr.scalar.SPDE} to the equation \eqref{eq:fSPDE} with $\hurst=1/2$, where $\sum_{\ell=0}^{\infty}(2\ell+1)(\vT+(\ell+1)^{\alpha+\gamma} \vfBm)<\infty$. Let $t>0$ and $h>0$. Then,
	\begin{equation}
		\norm{\sol(t+h) - \sol(t)}{\Lppsph{2}{2}} \le C h,\label{eq:UB.Bm.L2err.sol}
	\end{equation}
	where the constant $C$ depends only on $\alpha$, $\gamma$, $t$, $\vT$ and $\vfBm$.
\end{corollary}

\section{Numerical examples}\label{sec:numer}
In this section, we show some numerical examples for the solution $\sol(t)$ of the fractional SPDE \eqref{eq:fSPDE}. Using a $2$-weakly isotropic Gaussian random field as the initial condition, we illustrate the convergence rates of the truncation errors and the mean square approximation errors of the Karhunen-Lo\`{e}ve approximations $\trsol(t)$ of the solution $\sol(t)$. We show the evolutions of the solution of the equation \eqref{eq:fSPDE} with CMB (cosmic microwave background) map as the initial random field.

\subsection{Gaussian random field as initial condition}\label{sec:numer.Gauss}
Let $\RF_{0}$ be the $2$-weakly isotropic Gaussian random field whose Fourier coefficient  $\Fcoe{(\RF_{0})}$ follows the normal distribution $\nrml(0,\APS')$ for each pair of $(\ell,m)$, where the variances 
\begin{equation}\label{eq:APS.numer}
	\APS' := \left\{\begin{array}{ll}
		\APS[0], & \ell=0,\\
		\APS/2, & \ell\ge1,
	\end{array}
	\right.
\end{equation}
with
\begin{equation*}
	\APS:=1/(1+\ell)^{2\smind+2}, \quad \smind>1,
\end{equation*}
see \eqref{eq:ran.cond} and \eqref{eq:vTc}.

The initial condition of the equation \eqref{eq:fSPDE} is $\solC(t_{0},\PT{x})$ given by \eqref{eq:sol.fr.stoch.Cauchy}. The fBm is given by \eqref{eq:fBmsph.real.expan} with variances
\begin{equation}\label{eq:vfBm.numer}
	\vfBm:=1/(1+\ell)^{2\smind+2}, \quad \smind>1.
\end{equation}

By \citep[Section~4]{LaSc2015}, the random fields $\RF_{0}$ with angular power spectrum $\APS$ in \eqref{eq:APS.numer} and $\fBmsph(t)$ with variances $\vfBm$ in \eqref{eq:vfBm.numer} at $t\ge0$ are in Sobolev space $\sob{2}{r}{2}$ $\Pas$, and thus can be represented by a continuous function on $\sph{2}$ almost surely. This enables numerical implementation.

To obtain numerical results, we use $\trsol[L_{0}](t)$ with $L_{0}=1000$ as a substitution of the solution $\sol(t)$ in Theorem~\ref{thm:sol.fr.SPDE.scalar} to the equation \eqref{eq:fSPDE}. The truncated expansion $\trsol(t)$ given in Definition~\ref{defn:KL.approx} is computed using the fast spherical Fourier transform \citep{KeKuPo2007,RoTy2006}, evaluated at $\Neval=12,582,912$ HEALPix (Hierarchical Equal Area isoLatitude Pixezation) points\footnote{\url{http://healpix.sourceforge.net}} on $\sph{2}$, the partition by which is equal-area, see \citep{Gorski_etal2005}. Then the (squared) mean $L_{2}$-errors are evaluated by
\begin{align*}
    \norm{\trsol(t)-\sol(t)}{\Lppsph{2}{2}}^{2}
    &=\expect{\normb{\trsol(t)-\sol(t)}{\Lp{2}{2}}^{2}}\\
    &= \expect{\int_{\sph{2}} \bigl|\trsol(t,\PT{x}) - \sol(t,\PT{x})\bigr|^{2} \IntDiff[]{x}}\notag\\
    &\approx \expect{\frac{1}{\Neval}\sum_{i=1}^{\Neval} \bigl(\trsol(t,\peval{i}) - \sol(t,\peval{i})\bigr)^{2}}\notag\\
    &\approx \frac{1}{\sampnum\Neval}\sum_{n=1}^{\sampnum}\sum_{i=1}^{\Neval} \bigl(\trsol(t,\sampdis_{n},\peval{i}) - \sol(t,\sampdis_{n},\peval{i})\bigr)^{2},
\end{align*}
where the third line discretizes the integral on $\sph{2}$ by the HEALPix points $\peval{i}$ with equal weights $1/\Neval$, and the last line approximates the expectation by the mean of $\sampnum$ realizations.

In a similar way, we can estimate the mean square approximation error between $\trsol(t+h)$ and $\trsol(t)$ for $t\ge0$ and $h>0$.

For each realization and given time $t$, the fractional stochastic integrals in \eqref{eq:fSI.0t} in the expansion of $\trsol(t)$ in \eqref{eq:trsol} are simulated as independent, normally distributed random variables with means zero and variances $\vfBm$ in \eqref{eq:var.int.fBm}.

Using the fast spherical Fourier transform, the computational steps for $\sampnum$ realizations of $\trsol(t)$ evaluated at $\Neval$ points are $\bigo{}{\sampnum\Neval\sqrt{\log\Neval}}$.

The simulations were carried out on a desktop computer with Intel Core i7-6700 CPU @ 3.47GHz with 32GB RAM under the Matlab R2016b environment. 

Figure~\ref{fig:L2err_H05_a08_g08} shows the mean $L_{2}$-errors of $\sampnum=100$ realizations of the truncated Karhunen-Lo\`{e}ve solution $\trsol(t)$ with degree $L$ up to $800$ from the approximated solution $\trsol[L_{0}](t)$, of the fractional SPDE \eqref{eq:fSPDE} with the Brownian motion $\fBmsph[1/2](t)$, for $\smind=1.5$ and $2.5$ and $(\alpha,\gamma)=(0.8,0.8)$ at $t=t_{0}=10^{-5}$.

Figure~\ref{fig:L2err_H08_a05_g05} shows the mean $L_{2}$-errors of $\sampnum=100$ realizations of the truncated Karhunen-Lo\`{e}ve expansion $\trsol(t)$ with degree $L$ up to $800$ from the approximated solution $\trsol[L_{0}](t)$, of the fractional SPDE with the fractional Brownian motion $\fBmsph(t)$ with Hurst index $\hurst=0.8$, for $\smind=1.5$ and $2.5$ and $(\alpha,\gamma)=(0.5,0.5)$ at $t=t_{0}=10^{-5}$.

The green and yellow points in each picture in Figure~\ref{fig:L2err} show the $L_{2}$-errors of $\sampnum=100$ realizations of $\sol_{L}(t)$. For each $L=1,\dots,1000$, the sample means of the $L_{2}$-errors for $r=1.5$ and $2.5$ are shown by the red triangle and the brown hexagon respectively. In the log-log plot, the blue and cyan straight lines which show the least squares fitting of the mean $L_{2}$-errors give the numerical estimates of convergence rates for the approximation of $\sol_{L}(t)$ to $\sol_{L_{0}}(t)$. 

The results show that the convergence rate of the mean $L_{2}$-error of $\sol_{L}(t)$ is close to the theoretical rate $L^{-\smind}$ ($\smind=1.5$ and $2.5$) for each triple of $(\hurst,\alpha,\gamma)=(0.5,0.8,0.8)$ and $(0.8,0.5,0.5)$. This indicates that the Hurst index $\hurst$ for the fBm $\fBmsph(t)$ and the index $\alpha,\gamma$ for the fractional diffusion operator $\psi(-\LBo)$ have no impact on the convergence rate of the $L_{2}$-error of the truncated solution.  

\begin{figure}
  \centering
  \begin{minipage}{\textwidth}
  \centering
  \begin{minipage}{\textwidth}
  \begin{minipage}{0.485\textwidth}
  \centering
  \includegraphics[trim = 0mm 0mm 0mm 0mm, width=0.9\textwidth]{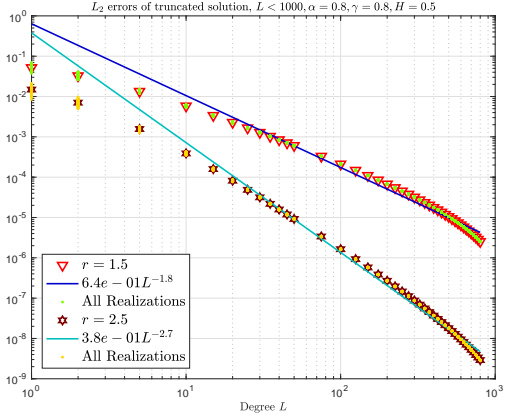}\\
  \subcaption{$\hurst=0.5$, $\alpha=0.8$, $\gamma=0.8$~~~~~}\label{fig:L2err_H05_a08_g08}
  \end{minipage}
  \begin{minipage}{0.485\textwidth}
  \centering
  \includegraphics[trim = 0mm 0mm 0mm 0mm, width=0.9\textwidth]{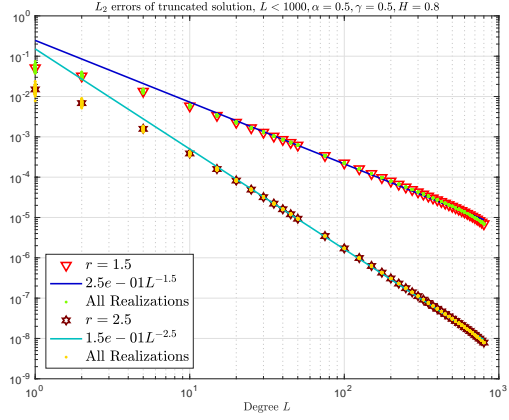}\\
  \subcaption{$\hurst=0.8$, $\alpha=0.5$, $\gamma=0.5$~~~~~}\label{fig:L2err_H08_a05_g05}
  \end{minipage}
  \end{minipage}\\[2mm]
  \begin{minipage}{0.8\textwidth}
\caption{\scriptsize (a)--(b) show the mean truncated $L_{2}$-errors of the Karhunen-Lo\`{e}ve approximations $\sol_{L}(t)$ with degree $L$ up to $800$ at $t=t_0=10^{-5}$, for $(\hurst,\alpha,\gamma)=(0.5,0.8,0.8)$ and $(\hurst,\alpha,\gamma)=(0.8,0.5,0.5)$ respectively. The $X$-axis is the degree $L$ and the $Y$-axis indicates the $L_{2}$-errors.
}\label{fig:L2err}
\end{minipage}
\end{minipage}
\end{figure}

Figure~\ref{fig:MQV} shows the mean square approximation errors of $\sampnum=100$ realizations of the truncated Karhunen-Lo\`{e}ve expansion $\trsol(t+h)$ from $\trsol(t)$ with degree $L=1000$ of the fractional SPDE with the fractional Brownian motion $\fBmsph(t)$ with Hurst index $\hurst=0.5$ and $\hurst=0.9$, for $\smind=1.5$ and $(\alpha,\gamma)=(0.8,0.8)$ at $t=t_{0}=10^{-5}$ and time increment $h$ ranging from $10^{-7}$ to $10^{-1}$.

The green points in each picture in Figure~\ref{fig:MQV} show the (sample) mean square approximation errors of $\sampnum=100$ realizations of $\sol_{L}(t+h)$. The blue straight line which shows the least squares fitting of the mean square approximation errors gives the numerical estimate of the convergence rate for the approximation of $\trsol(t+h)$ to $\trsol(t)$ in time increment $h$. 

For $\hurst=0.9$, Figure~\ref{fig:MQV_H08_a05_g05} shows that the convergence rate of the mean square approximation errors of $\trsol(t+h)$ is close to the theoretical rate $h^{\hurst}$. 
For $\hurst=0.5$, Figure~\ref{fig:MQV_H05_a08_g05} shows that the convergence rate of the mean square approximation errors of $\trsol(t+h)$ is close to the  rate $h$ as Corollary~\ref{cor:MQV.Bm} suggests. The variance of the mean square approximation errors for $\hurst=0.5$ is larger than for $\hurst=0.9$ for given $h$.
This illustrates that the Hurst index $\hurst$ for the fBm affects the smoothness of the evolution of the solution of the fractional SPDE \eqref{eq:fSPDE} with respect to time $t$.  

Figures~\ref{fig:Gauss.KL0} and \ref{fig:Gauss.KL} illustrate realizations of the truncated solutions $X_{L_0}(t)$ and $\sol_{L}(t)$ with $L_0=1000$ and $L=800$ at $t=t_{0}=10^{-5}$, evaluated at $\Neval=12,582,912$ HEALPix points. Figure~\ref{fig:Gauss.KLerr} shows the corresponding pointwise errors between $X_{L_0}(t)$ and $\sol_{L}(t)$. It shows that the truncated solution $\trsol(t)$ has good approximation to the solution $\sol(t)$ and the pointwise errors which are almost uniform on $\sph{2}$ are very small compared to the values of $\trsol(t)$.

\begin{figure}
  \centering
  \begin{minipage}{\textwidth}
  \centering
  \begin{minipage}{\textwidth}
  \begin{minipage}{0.485\textwidth}
  \centering
  \includegraphics[trim = 0mm 0mm 0mm 0mm, width=0.9\textwidth]{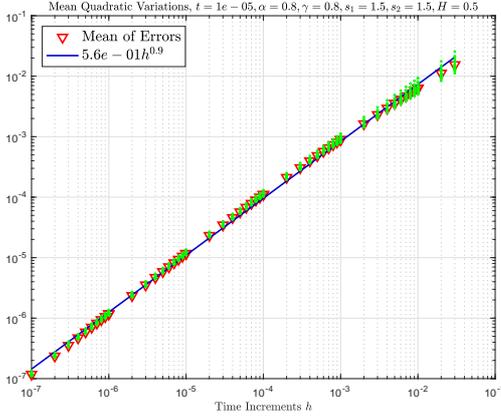}\\
  \subcaption{$\hurst=0.5$, $\alpha=0.8$, $\gamma=0.8$~~~~~}\label{fig:MQV_H05_a08_g05}
  \end{minipage}
  \begin{minipage}{0.485\textwidth}
  \centering
  \includegraphics[trim = 0mm 0mm 0mm 0mm, width=0.9\textwidth]{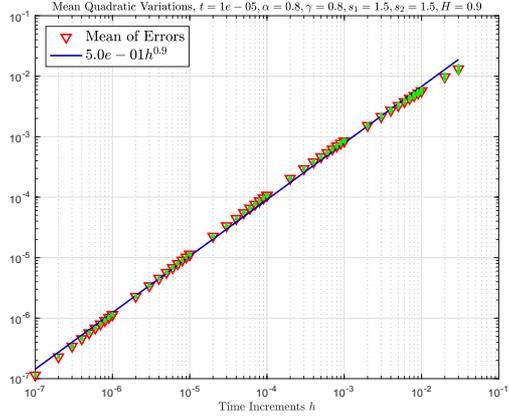}\\
  \subcaption{$\hurst=0.9$, $\alpha=0.8$, $\gamma=0.8$~~~~~}\label{fig:MQV_H08_a05_g05}
  \end{minipage}
  \end{minipage}\\[2mm]
  \begin{minipage}{0.8\textwidth}
\caption{\scriptsize (a)--(b) show the mean square approximation errors of the Karhunen-Lo\`{e}ve approximations $\sol_{L}(t)$ with degree $L=1000$ and $(\alpha,\gamma)=(0.8,0.8)$ at $t=t_0=10^{-5}$, for Hurst index $\hurst=0.5$ and $0.9$ respectively. The $X$-axis is the time increments $h$ ranging from $10^{-7}$ to $10^{-1}$.
}\label{fig:MQV}
\end{minipage}
\end{minipage}
\end{figure}

\begin{figure}
  \centering
  \begin{minipage}{\textwidth}
  \centering
  \begin{minipage}{\textwidth}
  \begin{minipage}{0.32\textwidth}
  \centering
  \includegraphics[trim = 15mm 0mm 0mm 0mm, width=1.09\textwidth]{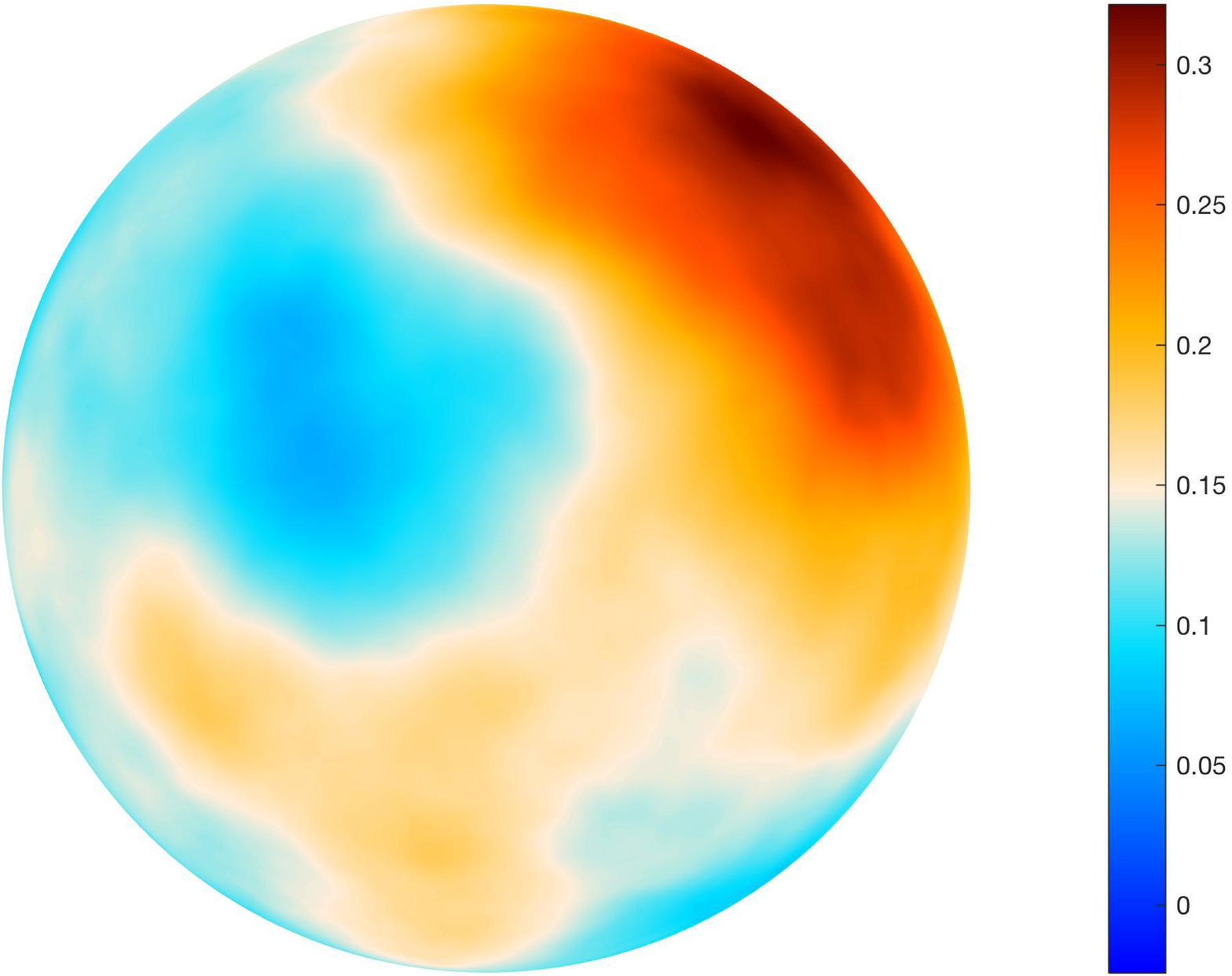}\\[-5mm]
  \subcaption{$X_{L_0}$, $L_0 = 1000$~~~~~}\label{fig:Gauss.KL0}
  \end{minipage}
  \hspace{0.0\textwidth}
  \begin{minipage}{0.32\textwidth}
  \centering
  \includegraphics[trim = 15mm 0mm 0mm 0mm, width=1.09\textwidth]{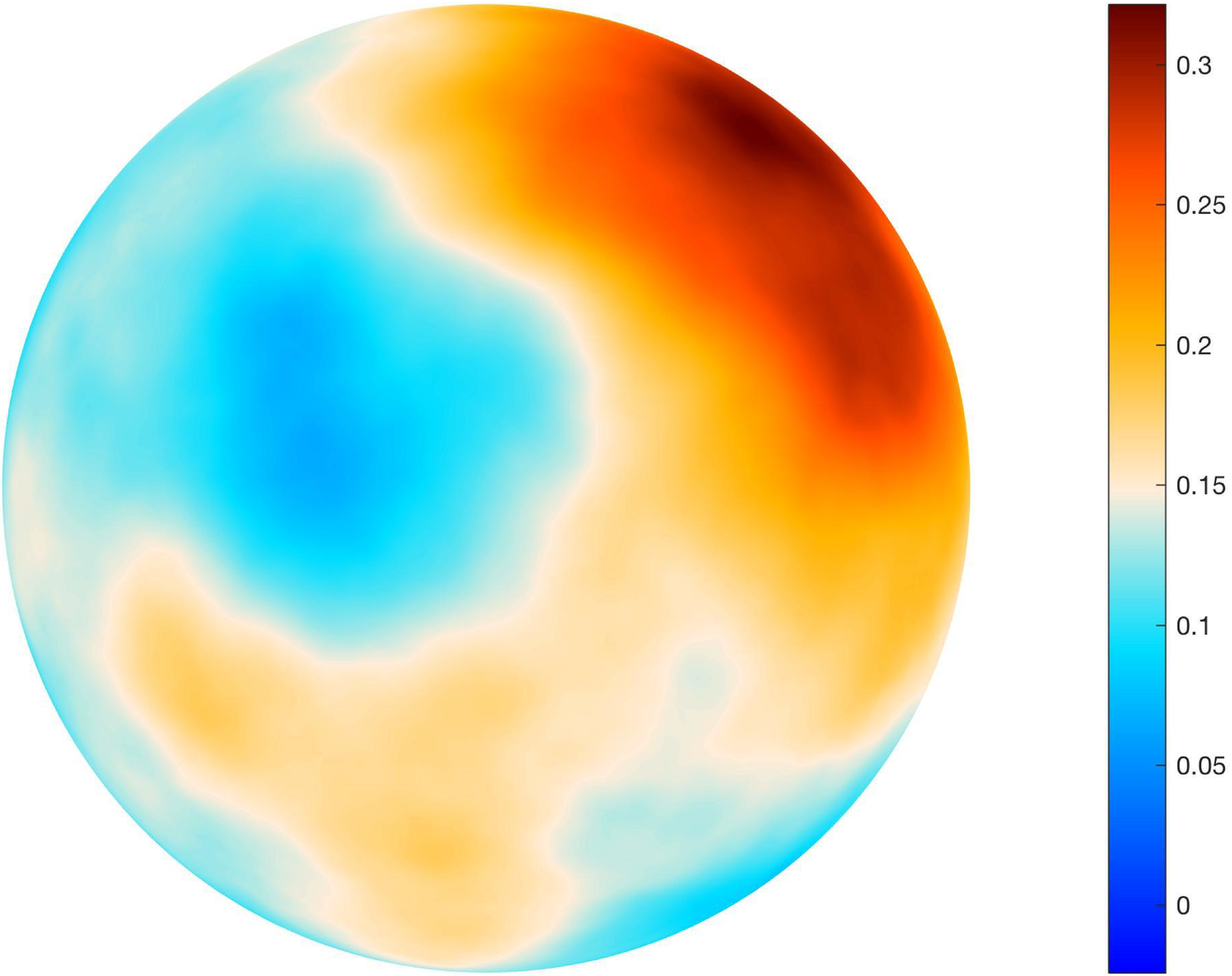}\\[-5mm]
  \subcaption{$X_{L}$, $L = 800$~~~~~}\label{fig:Gauss.KL}
  \end{minipage}
  \begin{minipage}{0.32\textwidth}
  \centering
  \includegraphics[trim = 15mm 0mm 0mm 0mm, width=1.15\textwidth]{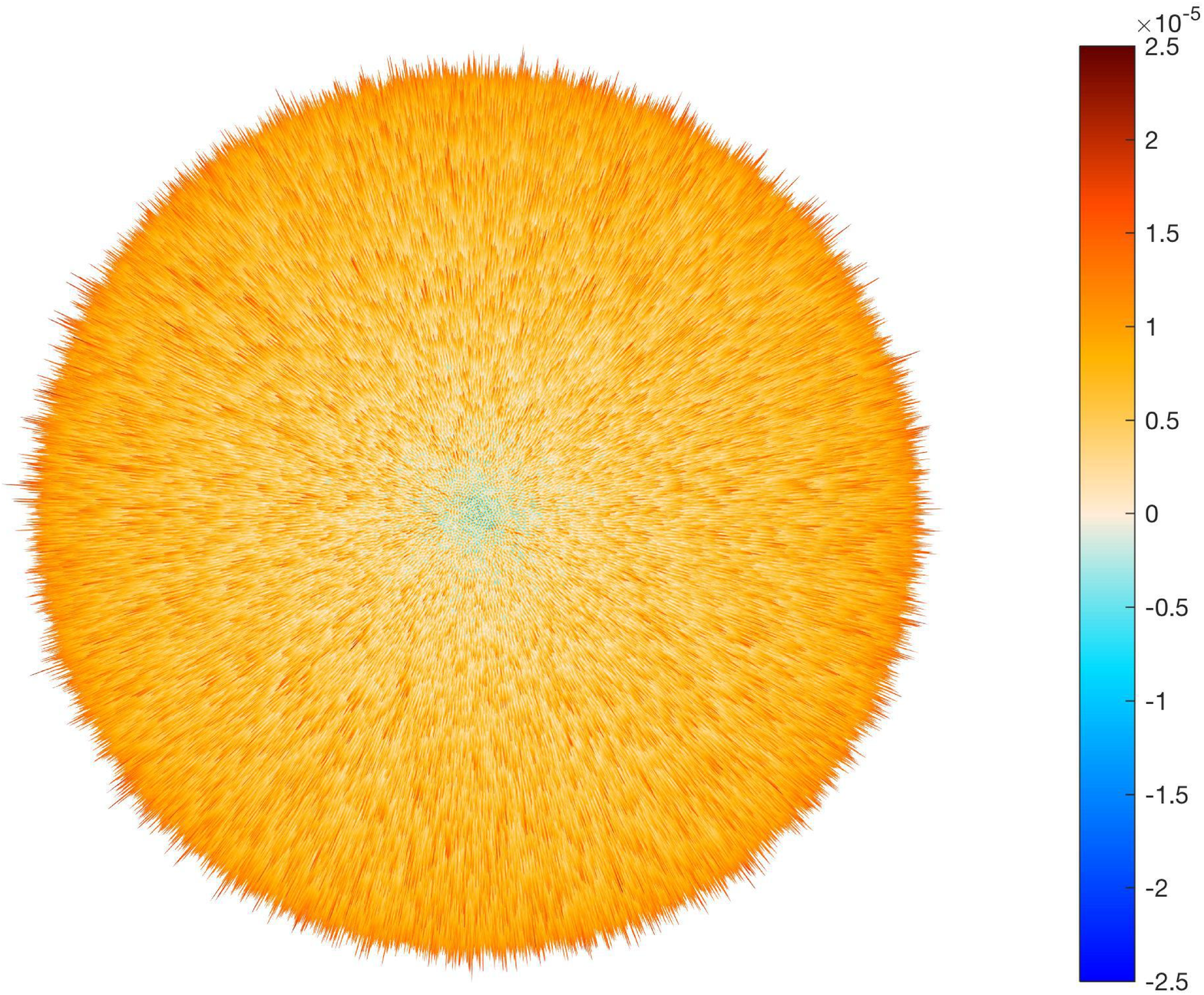}\\[-5mm]
  \subcaption{Errors $X_{L}-X_{L_0}$}\label{fig:Gauss.KLerr}
  \end{minipage}
  \end{minipage}
  \begin{minipage}{0.8\textwidth}
  \vspace{2mm}
\caption{\scriptsize (a) and (b) show realizations of the truncated Karhunen-Lo\`{e}ve expansion $\sol_{L}(t)$ with degree $L=800$ for the solution of the fractional SPDE, where $(\alpha,\gamma)=(0.5,0.5)$ and $t=t_{0}=10^{-5}$. (c) shows the pointwise errors of (b) from (a).
}\label{fig:X_a08_g05_no1}
\end{minipage}
\end{minipage}
\end{figure}

To understand further the interaction between this effect from the Hurst index $\hurst$ of fBm and the parameters $(\alpha, \gamma)$ from the diffusion operator, we generate realizations of $\trsol[1000](t)$ at time $t=t_0=10^{-5}$ for the cases $(\hurst, \alpha, \gamma)=(0.9, 2, -2)$, $(0.9, 1.5, -0.5)$ and $(0.9, 1, -0.5)$. These paths are displayed in Figure~\ref{fig:fBm_Gauss_H09}. We observe that these random fields have fluctuations (about the sample mean) of increasing size as $\alpha$ increases from $0.5$ in Figure~\ref{fig:X_a08_g05_no1} to $1$, $1.5$ then $2$ in Figure~\ref{fig:fBm_Gauss_H09}. In fact, the fluctuation is extreme when $\alpha +\gamma =0$. When $\alpha =2$, the density of the Riesz-Bessel distribution has sharper peaks and heavier tails as $\gamma \rightarrow -2$. As explained in \citep{AnMc2004}, the L\'{e}vy motion in the case $\alpha +\gamma =0$ is a compound Poisson process. The particles move through jumps, but none of the jumps is very large due to the parameter $\alpha =2$. Hence the distribution has finite moments of all orders.


\begin{figure}
  \centering
  \begin{minipage}{\textwidth}
  \centering
  \begin{minipage}{\textwidth}
  \begin{minipage}{0.32\textwidth}
  \centering
  \includegraphics[trim = 15mm 0mm 0mm 0mm, width=1.09\textwidth]{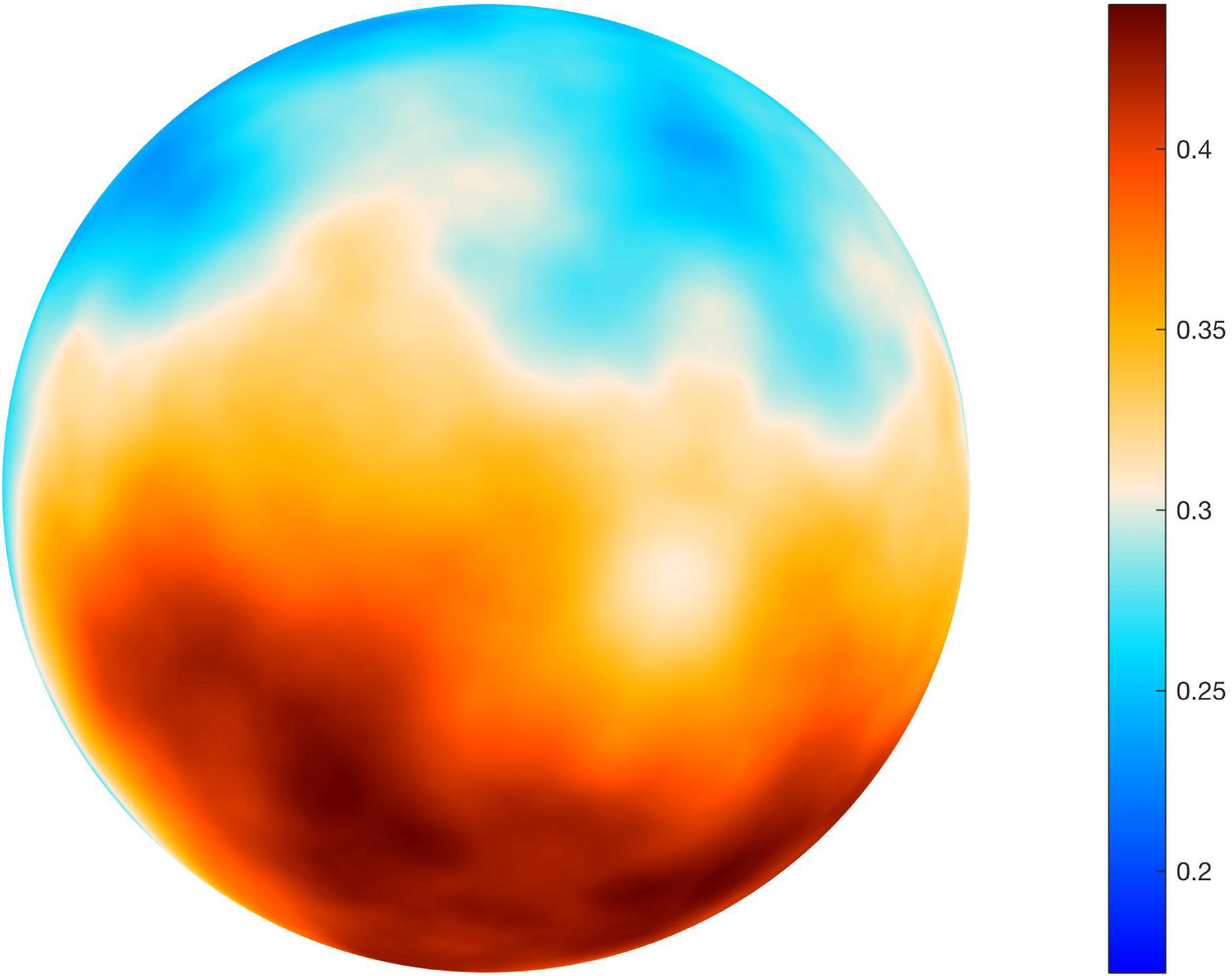}\\[-5mm]
  \subcaption{$X_{1000}(t)$, $\alpha=2$, $\gamma=-2$~~~~~}\label{fig:Gauss.H09_a2_g-2}
  \end{minipage}
  \hspace{0.0\textwidth}
  \begin{minipage}{0.32\textwidth}
  \centering
  \includegraphics[trim = 15mm 0mm 0mm 0mm, width=1.09\textwidth]{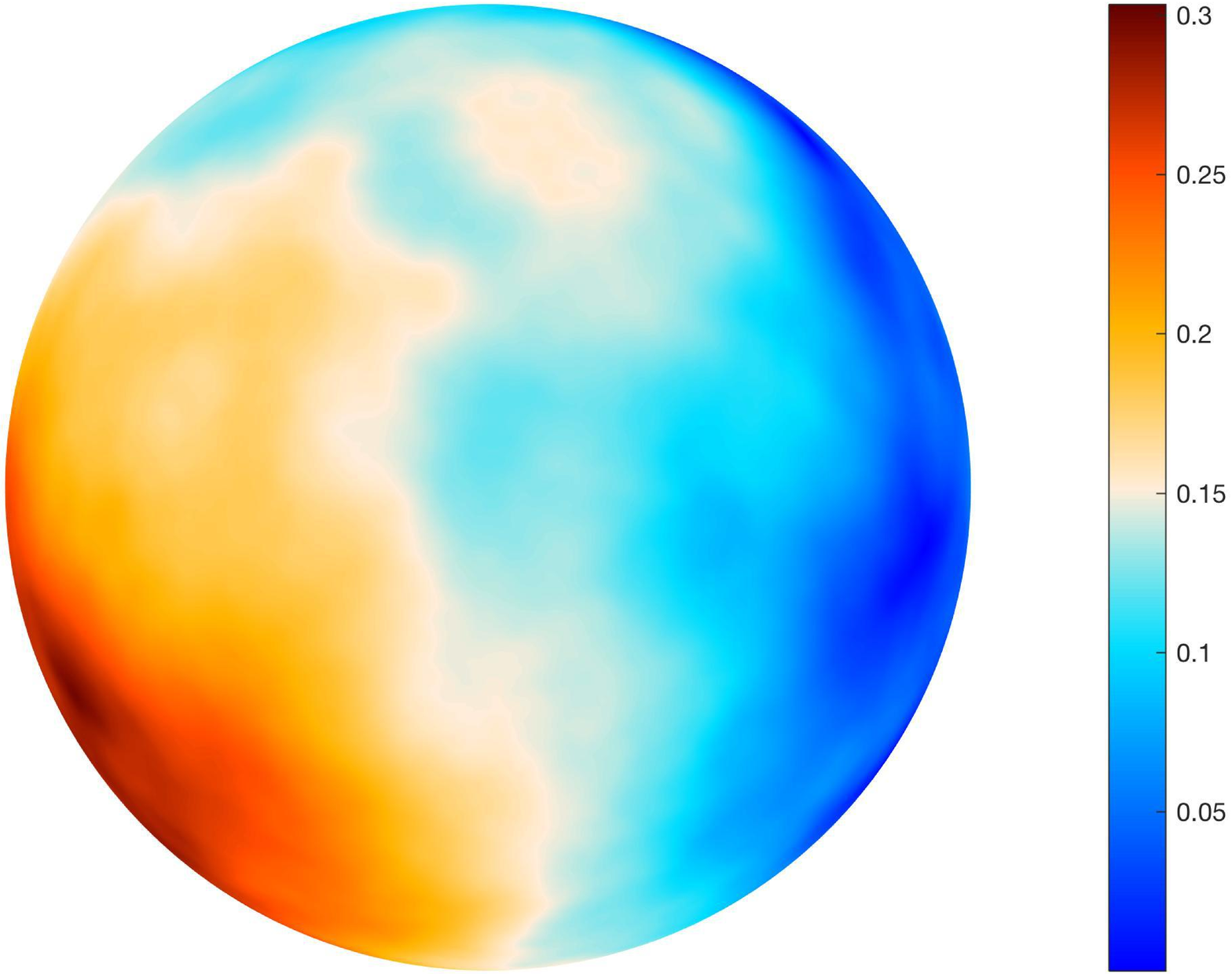}\\[-5mm]
  \subcaption{$X_{1000}(t)$, $\alpha=1.5$, $\gamma=-0.5$~~~~~}\label{fig:Gauss.H09_a15_g-05}
  \end{minipage}
  \begin{minipage}{0.32\textwidth}
  \centering
  \includegraphics[trim = 15mm 0mm 0mm 0mm, width=1.1\textwidth]{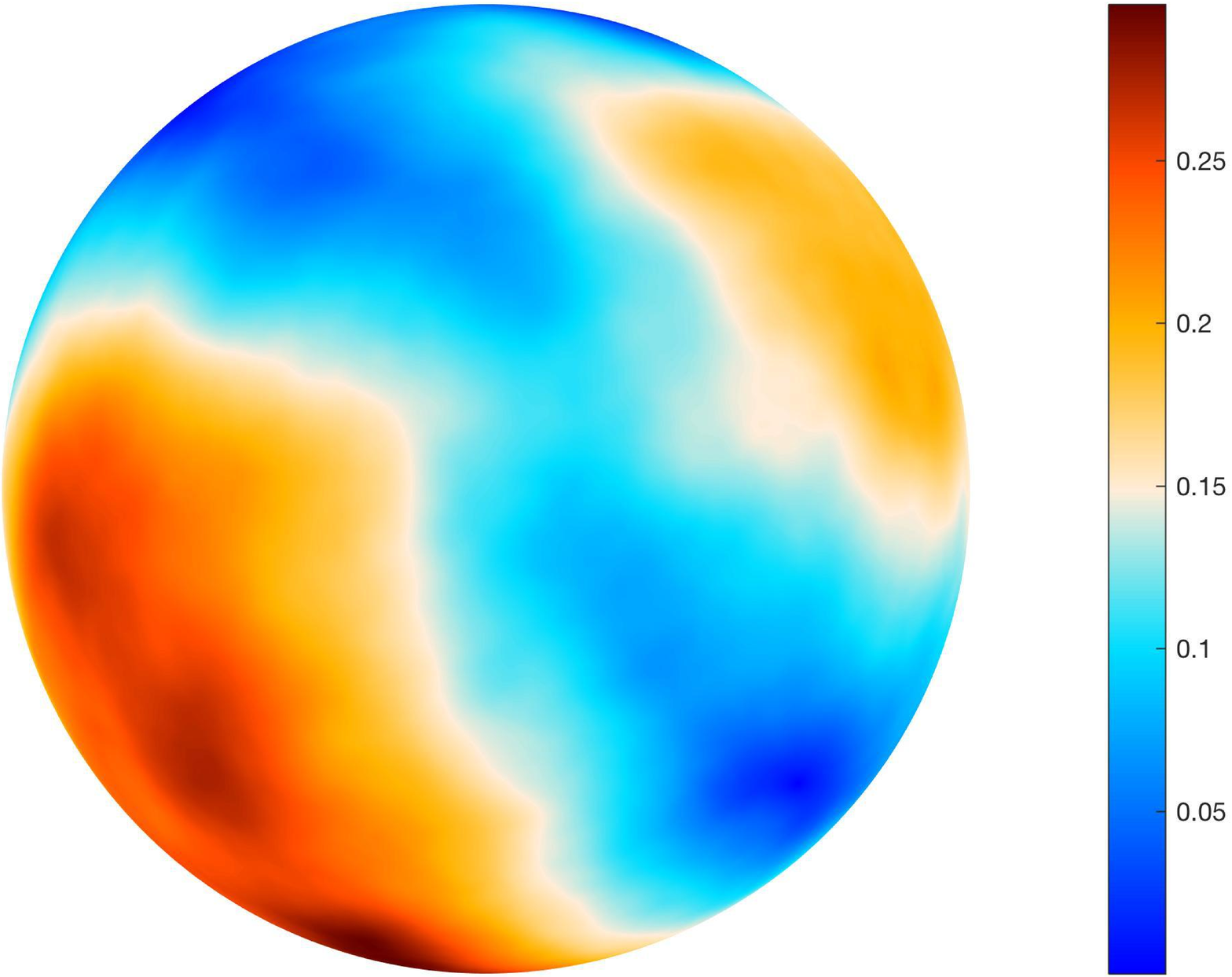}\\[-5mm]
  \subcaption{$X_{1000}(t)$, $\alpha=1$, $\gamma=-0.5$}\label{fig:Gauss.H09_a1_g-05}
  \end{minipage}
  \end{minipage}\\[2mm]
  \begin{minipage}{0.8\textwidth}
\caption{\scriptsize (a), (b) and (c) show realizations of the truncated Karhunen-Lo\`{e}ve expansion $\sol_{1000}(t)$ of the solution of the fractional SPDE for $(\alpha,\gamma)=(2,-2)$, $(1.5,-0.5)$ and $(1,-0.5)$, where $\hurst=0.9$ and $t=t_{0}=10^{-5}$.
}\label{fig:fBm_Gauss_H09}
\end{minipage}
\end{minipage}
\end{figure}

\subsection{CMB random field as initial condition}
The cosmic microwave background (CMB) is the radiation that was in equilibrium with the plasma of the early universe, decoupled at the time of recombination of atoms and free electrons. Since then, the electromagnetic wavelengths have been stretching with the expansion of the universe (for a description of the current standard cosmological model, see e.g. \citep{Dodelson2003}). The inferred black-body temperature shows direction-dependent variations of up to 0.1\%. The CMB map that is the sky temperature intensity of CMB radiation can be modelled as a realization of a random field $\cmbRF$ on $\sph{2}$. 
Study of the evolution of CMB field is critical to unveil important properties of the present and primordial universe \citep{PiScWh2000,Planck2016I}.

From the factors that are exponential in $t$ in (4.5), it is apparent that within the normalized form of fractional SPDE that has been considered in \eqref{eq:fSPDE}, it is assumed that the variables $X$, $\PT{x}$ and $t$ have already been rescaled to make them dimensionless. Now we consider the case that the variable $X$ is replaced by temperature $\widetilde{\sol}$ for which the dimensions of measurement have traditionally been written as $\Theta$ (see e.g. Fulford and Broadbridge \cite{FuBr2002}). In its dimensional form, the fractional SPDE is
\begin{equation*}
	\IntD{\widetilde{\sol}(t,\PT{x})}+\frac{1}{t_s}(-\ell_s^2\LBo)^{\alpha/2} (I-\ell_s^2\LBo)^{\gamma/2}\widetilde{\sol}(t,\PT{x})\IntD{t}=\sigma(\hspace{-0.5mm}\IntD{t})^{\hurst}\fBm(1,\PT{x})
\end{equation*}
where $\LBo$ has dimension $[L]^{-2}$, where $[L]$ stands for the unit of length,  
necessitating the appearance of a length scale $\ell_s$ and a time scale $t_s$, in order for each term on the left hand side to have dimensions of temperature. The right side also has dimension of temperature, since for this fractional Brownian motion of temperature, $\sigma= \Theta T^H$ and the random function $B^H(1,{\bf x}) $ has variance 1. In this work we are considering $\ell_s$ to be the radius of the sphere which is thereafter assumed $1$. Then one may choose dimensionless time $t^*=t/t_s$ and dimensionless temperature $\widetilde{\sol}^{*}(t)=\widetilde{\sol}(t)/\widetilde{\sol}_s$ with $\widetilde{\sol}_s=\sigma t_s^H$, so that the dimensionless variables satisfy the normalized equation \eqref{eq:fSPDE} with a normalized ($\sigma^*=1$) random forcing term. However, in physical descriptions it is sometimes convenient to choose other scales such as the mean CMB for temperature and the current age of the universe for time. In that case, the exponential factor $\exp(-\freigv t/t_s)$ for decay of amplitudes in (4.5) would involve  a time scale that is no longer set to be unity.

We adopt the \emph{conformal time} 
\begin{equation*}
	\eta := \eta(t) := \int_{0}^{t}\frac{d{t'}}{a(t')},\quad t\ge0
\end{equation*}
to replace the  cosmic time $t$ in the equations, where $a(t)$ is the {scale factor} of the universe at time $t$.
By using $\eta$ rather than $t$ as the time coordinate of evolution in the fractional SPDE, the decrease of spectral amplitudes is in better agreement with current physical theory.

We are concerned with the evolution of CMB before recombination that occurred at time $\eta=\eta_*$. If the present time is set to be $\eta_0=1$, then $\eta_{*}\approx 2.735\times 10^{-5}$. (We use $13.82$ billion years, estimated by Planck 2015 results \cite{Planck2015XIII}, as the age of the universe.) In the epoch of the radiation-dominated plasma universe, $a(t)$ is proportional to $t^{1/2}$, and thus $\eta$ is proportional to $t^{1/2}$, therefore proportional to $a(t)$. Then, temperature varies in proportion to $1/t$ (e.g. refer to Weinberg \cite{Weinberg2008}). Since the temperature is relatively uniform, all the amplitudes will vary approximately as $1/t$, and the peaks in the power spectrum will vary as $1/t^2$ and thus in proportion to $1/\eta^{4}$ (this does not take into account the motion of remnant acoustic waves). A large fraction  of the radiation-dominated epoch passes by as $\eta$ changes by $10^{-5}$.  In the following experiment, we examine the evolution by the fractional stochastic PDE for times a little beyond $\eta_*$, as if the radiation energy remained dominant.

The time scale $t_{s}$ can be estimated as follows. For $\eta\ge\eta_{*}$, let $\Delta\eta:=\eta-\eta_{*}$. Using the CMB map (in Figure~\ref{fig:CMB.map}) at recombination time $\eta_{*}$ as the initial condition, the scale $t_{s}$ between the increment of conformal time $\Delta \eta$ and the evolution time $t$ of the normalized equation \eqref{eq:fSPDE} can be determined by matching the ratio of magnitudes of angular power spectra of the evolutions $\widetilde{\sol}(\Delta\eta)=\sol(\Delta\eta/t_{s})$ at $\Delta \eta = \Delta\eta_{1}$ and $\Delta\eta_{2}$ and at degree $\ell$, that is,
\begin{equation}\label{eq:ratio.APS}
	\frac{\APS(\Delta\eta_{2}/t_{s})}{\APS(\Delta\eta_{1}/t_{s})} = \left(\frac{\eta_{1}}{\eta_{2}}\right)^{4}.
\end{equation}
By \eqref{eq:orth.Fcoe.RF} and the observation that $\exp\bigl(-\freigv(t+t_0)\bigr)\Fcoe{(\RF_{0})}\shY$ is the dominating term in \eqref{eq:sol.fr.scalar.SPDE}, \eqref{eq:ratio.APS} can be approximated by
\begin{equation*}
	\left(\frac{\eta_{1}}{\eta_{2}}\right)^{2} \approx \frac{a_{\ell m}(\Delta\eta_{2}/t_{s})}{a_{\ell m}(\Delta\eta_{1}/t_{s})} \approx \frac{\exp\bigl(-\freigv(\Delta\eta_{2}/t_{s}+t_{0})\bigr)}{\exp\bigl(-\freigv(\Delta\eta_{1}/t_{s}+t_{0})\bigr)} = e^{-\freigv (\eta_{2}-\eta_{1})/t_{s}},
\end{equation*}
where $a_{\ell m}(t)$ is the Fourier coefficient of $\sol(t)$ at degree $(\ell,m)$.
This gives
\begin{equation}\label{eq:time.scale}
	t_{s} := t_{s}(\ell) \approx \frac{(\eta_{2}-\eta_{1})\freigv}{2\ln (\eta_{2}/\eta_{1})}.
\end{equation}

In the experiment, we use CMB data from Planck 2015 results, see \citep{Planck2016I}.
The CMB data are located on $\sph{2}$ at HEALPix points as we used in Section~\ref{sec:numer.Gauss}.
Figure~\ref{fig:CMB.map} shows the CMB map at $N_{\rm side}=1024$ at $10$ arcmin resolution with $12\times 1024^2=12,582,912$ HEALPix points, see \citep{Planck2016IX}. It is computed by SMICA, a component separation method for CMB data processing, see \citep{CaLeDeBePa2008}.

\begin{figure}[th]
 \centering
  \begin{minipage}{\textwidth}
  \centering
\begin{minipage}{0.6\textwidth}
\centering
  \includegraphics[trim = 40mm 30mm 30mm 20mm, width=\textwidth]{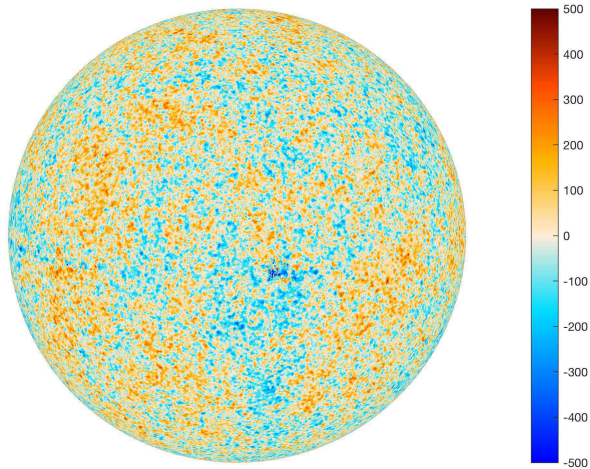}\\
  \begin{minipage}{0.8\textwidth}
  \caption{\scriptsize CMB map at $12,582,912$ HEALPix points} \label{fig:CMB.map}
  \end{minipage}
\end{minipage}
\end{minipage}
\end{figure}

We use the angular power spectrum  of CMB temperature intensities in Figure~\ref{fig:CMB.map}, obtained by Planck 2015 results \cite{Planck2016XI}, as the initial condition (the angular power spectrum of the initial Gaussian random field) of Cauchy problem \eqref{eq:fSCauchy}. In the fractional SPDE \eqref{eq:fSPDE}, we take $(\alpha,\gamma, H)=(0.5,0.5,0.9)$ so that the fractional diffusion operator acts as the drifting term within an evolution equation that is wave-like. The long-tailed stochastic fluctuations take account of the turbulence within the hot plasma.

For $\ell\ge0$, let $D_{\ell}:=\ell(\ell+1)\APS/(2\pi)$ be the scaled angular power spectrum.
We take $\eta_{1}=1.001\eta_{*}\approx 2.743\times10^{-5}$ and $\eta_{2}=1.1\eta_{*}\approx 3.014\times10^{-5}$, and then the first (highest) humps (at $\ell=219$) of the scaled angular power spectra $D_{\ell}(\Delta\eta_{1}/t_{s})$ and $D_{\ell}(\Delta\eta_{2}/t_{s})$ are expected to be $99.6\%$ and $68.3\%$ of that at recombination time $\eta_{*}$ respectively. The time scale can be determined by \eqref{eq:time.scale} taking $\ell=219$, then, $t_{s}=t_{s}(219)\approx0.0032$.

The picture in Figure~\ref{fig:H09_t8.68e-06} shows a realization of the solution $\sol_{1000}(\Delta\eta_{1}/t_{s})$ at an early conformal time $\eta_{1}=1.001 \eta_{*}$ which is similar to the original CMB map in Figure~\ref{fig:CMB.map}. The picture of Figure~\ref{fig:H09_t0.000868} shows a realization of the solution $\sol_{1000}(\Delta\eta_{2}/t_{s})$ at $\eta_{2}=1.1\eta_{*}$.

\begin{figure}[th]
  \centering
  \begin{minipage}{0.9\textwidth}
  \centering
  \begin{minipage}{\textwidth}
  \begin{minipage}{0.485\textwidth}
  \centering
  \includegraphics[trim = 40mm 50mm 30mm 20mm, width=1.02\textwidth]{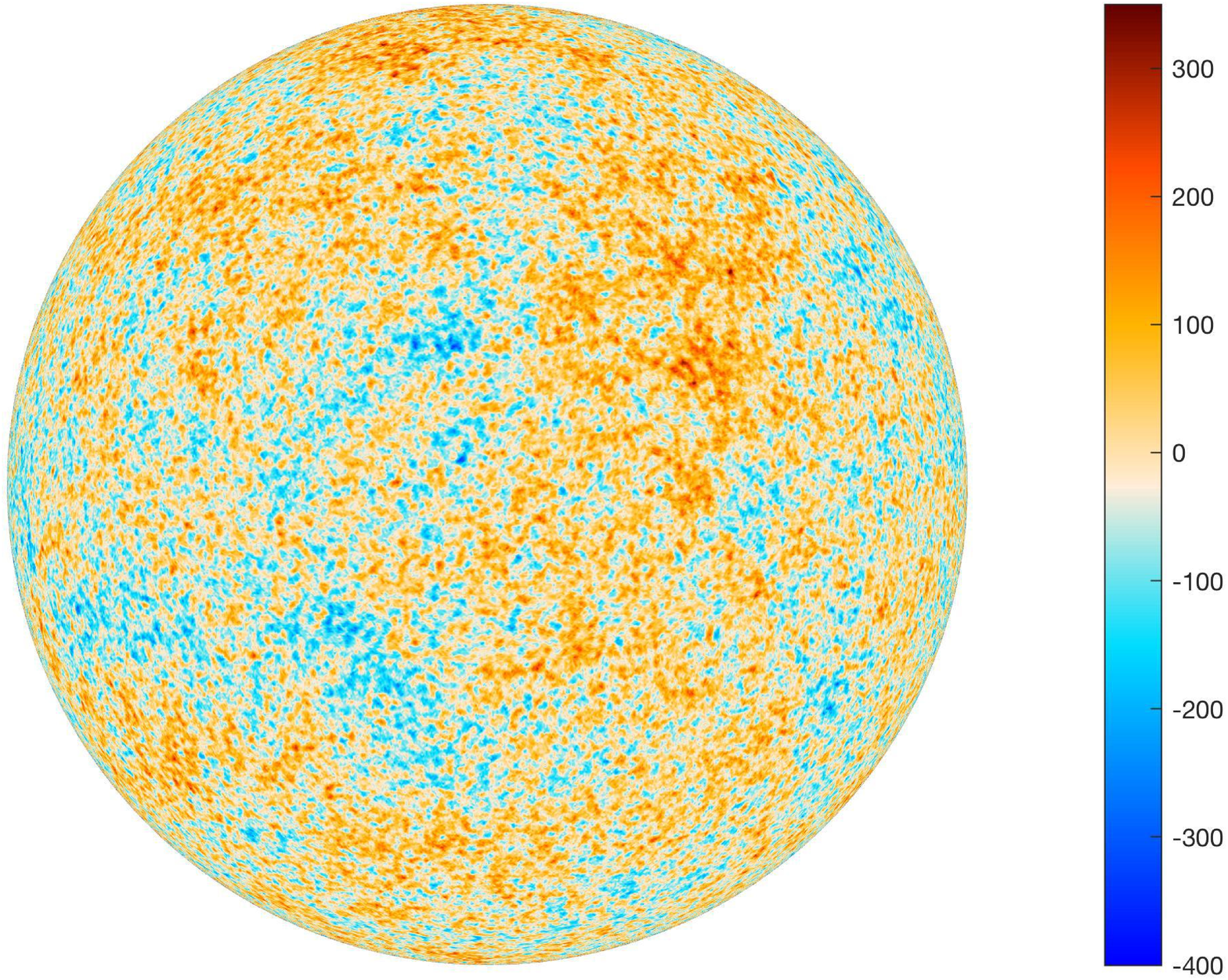}\\
  \subcaption{$X_{1000}(\Delta\eta_{1}/t_{s})$, $\eta_{1}=1.001\eta_{*}$~~~~~}\label{fig:H09_t8.68e-06}
  \end{minipage}
  \hspace{0.0\textwidth}
  \begin{minipage}{0.485\textwidth}
  \centering
  \includegraphics[trim = 40mm 50mm 30mm 20mm, width=1.02\textwidth]{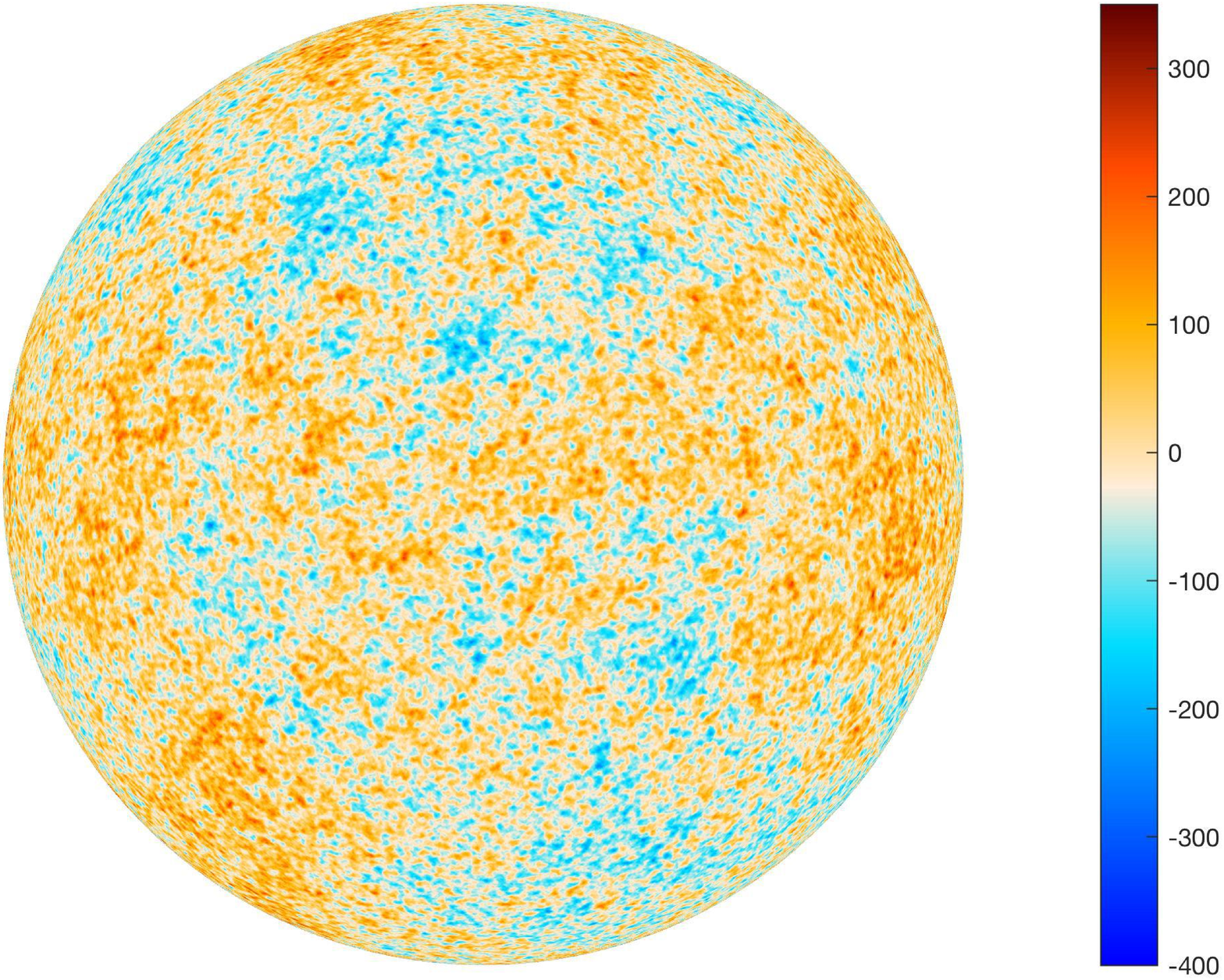}\\
  \subcaption{$X_{1000}(\Delta\eta_{2}/t_{s})$, $\eta_{2}=1.1\eta_{*}$~~~~~}\label{fig:H09_t0.000868}
  \end{minipage}
  \end{minipage}
  \end{minipage}\\[2mm]
  \begin{minipage}{0.8\textwidth}
\caption{\scriptsize (a)--(b) show realizations of the truncated Karhunen-Lo\`{e}ve expansion $\sol_{L}(\Delta\eta/t_{s})$ with degree $L=1000$ at $\eta=1.001\eta_{*}\approx 2.743\times 10^{-5}$ and $1.1\eta_{*}\approx 3.014\times10^{-5}$ for the solution of the fractional SPDE with Hurst index $\hurst=0.9$, and the CMB angular power spectrum $\APS$ at recombination time $\eta_{*}\approx 2.735\times 10^{-5}$ as the angular power spectrum for $\RF_{0}$ in the initial condition, where $(\alpha,\gamma)=(0.5,0.5)$, $t_{0}=10^{-7}$ and time scale $t_{s}\approx 0.0032$.
}\label{fig:CMB.evol}
\end{minipage}
\end{figure}

By Lemma~\ref{lem:orth.Fcoe.RF}, we estimate the angular power spectrum of solution $\sol(t)$ by taking the mean of the squares of the Fourier coefficients $a_{\ell m}(t)$ over integer orbital index $m\in\{-\ell,-\ell+1,\cdots,\ell\}$ and over $N$ realizations:
\begin{equation*}
	\APS(t) = \frac{1}{2\ell+1}\sum_{m=-\ell}^{\ell}\expect{|a_{\ell m}(t)|^{2}} \approx \frac{1}{N(2\ell+1)}\sum_{n=1}^{N}\sum_{m=-\ell}^{\ell}a_{\ell m}(t,\sampdis_{n}).
\end{equation*}

The reddish brown curve in Figure~\ref{fig:CMB.APS} shows the scaled angular power spectrum $D_{\ell}$, $\ell\le1000$, of CMB map at recombination time $\eta_{*}$.
The dots in orange around the reddish brown curve show the estimated the angular power spectrum $D_{\ell}(\Delta\eta_{1}/t_{s})$ of the solution $\sol_{1000}(\Delta\eta_{1}/t_{s})$ at evolution time $\Delta\eta_1 = 0.001\eta_*$ from the recombination time $\eta_{*}$, by taking the sample mean of $N=100$ realizations. It is observed that $D_{\ell}(0.001\eta_{*}/t_{s})$ changes little from those of CMB at $\eta_{*}$. 
The dots in blue show the estimated angular power spectrum $D_{\ell}(\Delta\eta_{2}/t_{s})$ of the solution $\sol_{1000}(\Delta\eta_{2}/t_{s})$ at $\Delta\eta_{2}=0.1\eta_{*}$ from $\eta_{*}$. The first hump of $D_{\ell}(0.1\eta_{*}/t_{s})$ (which appears at $\ell=210$) is about $69.8\%$ of that at the recombination time. This is consistent with the theoretical estimation $68.3\%$.

Figure~\ref{fig:CMB.APS} also shows the estimated angular power spectrum for one realization of $\sol_{1000}(\Delta\eta/t_{s})$ at $\eta=\eta_{1}$ and $\eta_{2}$ in light blue and red points. They have large noises around the sample means of $D_{\ell}$ (in orange and blue) when degree $\ell\le 400$.

\begin{figure}[th]
 \centering
  \begin{minipage}{\textwidth}
  \centering
\begin{minipage}{0.75\textwidth}
\centering
  \includegraphics[trim = 0mm 0mm 0mm 0mm, width=0.94\textwidth]{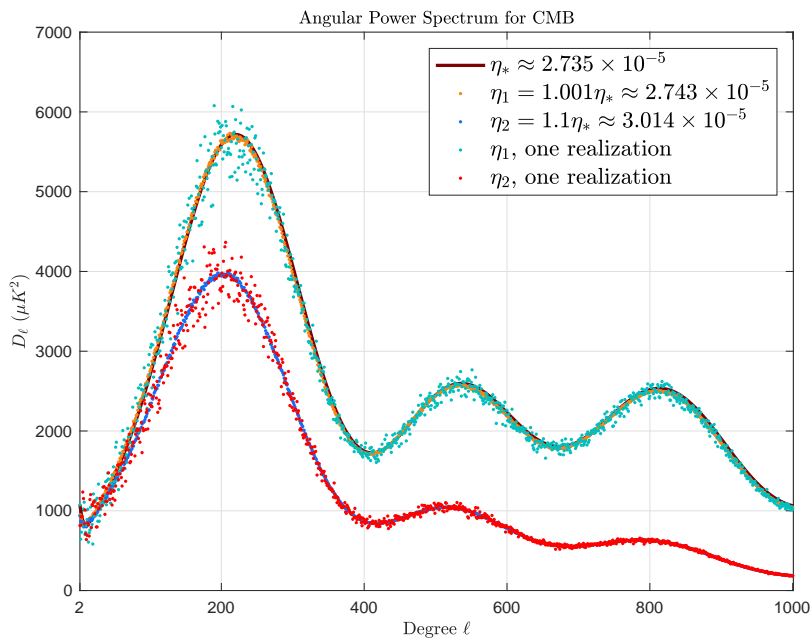}\\[2mm]
  \begin{minipage}{0.8\textwidth}
  \caption{\scriptsize Scaled angular power spectrum $D_\ell$ for CMB data at $\eta_{*}\approx 2.735\times10^{-5}$, and sample means of $D_{\ell}$ over $N=100$ realizations at $\eta_{1}=1.001\eta_{*}\approx 2.743\times10^{-5}$ (in orange) and $\eta_{2}=1.1\eta_{*}\approx 3.014\times10^{-5}$ (in blue), and estimated $D_{\ell}$ of one realization of $\sol_{1000}(\Delta\eta/t_{s})$ at $\eta=\eta_{1}$ (in light blue) and $\eta_{2}$ (in red)} \label{fig:CMB.APS}
  \end{minipage}
  \end{minipage}
\end{minipage}
\end{figure}


Figures~\ref{fig:CMB.APS} and \ref{fig:CMB.evol} illustrate that the solution of the fractional SPDE \eqref{eq:fSPDE} can be used to explore a possible forward evolution of CMB.

 Cosmological data typically have correlations over space-like separations, an imprint of quantum fluctuations and rapid inflation immediately after the big bang \citep{Guth1997}, followed by acoustic waves through the primordial ball of plasma. Fields at two points with space-like separation cannot be simultaneously modified by  evolution processes that obey the currently applicable laws of relativity. Therefore it is inappropriate to apply Brownian motion and standard diffusion models, with consequent unbounded propagation speeds, over cosmological distances. Although the phenomenological fractional SPDE models considered here are not relativistically invariant, they are a relatively simple device of maintaining long-range correlations.

\appendix


\section*{Acknowledgements}
This research was supported under the Australian Research Council's \emph{Discovery Project} DP160101366 and was supported in part by the La Trobe University DRP Grant in Mathematical and Computing Sciences.
We are grateful for the use of data from the Planck/ESA mission, downloaded from the Planck Legacy Archive. Some of the results in this paper have been derived using the HEALPix \citep{Gorski_etal2005}. This research includes extensive computations using the Linux computational cluster Raijin of the National Computational Infrastructure (NCI), which is supported by the Australian Government and La Trobe University.
The authors would thank Zdravko Botev for his helpful discussion on simulations of fractional Brownian motions. The authors also thank Ming Li for helpful discussion.


\appendix \section{Proofs}

\section*{Proof of Section~\ref{sec:fBm}}
\begin{proof}[Proof of Proposition~\ref{thm:BD.Int.fBmsph}]
	By Definition~\ref{defn:IntD.fBmsph}, Parseval's identity for the orthonormal basis $\shY$ and \citep[Theorem~1.1]{MeMiVa2001},
	\begin{align*}
		\expect{\norm{\int_{s}^{t}g(u)\IntD{\fBmsph(u)}}{\Lp{2}{2}}^{2}} 
		&= \sum_{\ell=0}^{\infty}\sum_{m=-\ell}^{\ell}\expect{\left|\int_{s}^{t}g(u)\IntD{\fBm_{\ell m}(u)}\right|^{2}}\\
		&\le C_{\hurst}\: \left(\int_{s}^{t}|g(u)|^{\frac{1}{\hurst}}\IntD{u}\right)^{2\hurst} \sum_{\ell=0}^{\infty}(2\ell+1)\vfBm,
	\end{align*}
	thus completing the proof.
\end{proof}

\section*{Proofs of Section~\ref{sec:fSPDE}}
\begin{proof}[Proof of Proposition~\ref{prop:isotr.sol.Cauchy}]
	For $t\ge0$, $\PT{x},\PT{y}\in\sph{2}$, by \eqref{eq:sol.fr.stoch.Cauchy},	
	\begin{align*}
		&\expect{\solC(t,\PT{x}) \solC(t,\PT{y})}\\
		&\quad= \expect{\sum_{\ell=0}^{\infty}\sum_{m=-\ell}^{\ell} e^{-\freigv t}\Fcoe{(\RF_{0})}\shY(\PT{x}) \sum_{\ell'=0}^{\infty}\sum_{m'=-\ell'}^{\ell'} e^{-\freigv[{\eigvm[\ell']}] t}\Fcoe[\ell' m']{(\RF_{0})}\shY[\ell',m'](\PT{y})}\\
		&\quad= \sum_{\ell=0}^{\infty}\sum_{\ell'=0}^{\infty}\sum_{m=-\ell}^{\ell}\sum_{m'=-\ell'}^{\ell'} e^{-\freigv t}e^{-\freigv[{\eigvm[\ell']}] t}\expect{\Fcoe{(\RF_{0})}\Fcoe[\ell' m']{(\RF_{0})}}\shY(\PT{x}) \shY[\ell',m'](\PT{y}).
	\end{align*}
	Since $\RF_{0}$ is a $2$-weakly isotropic Gaussian random field, by Lemma~\ref{lem:orth.Fcoe.RF},
\begin{align*}
		\expect{\solC(t,\PT{x}) \solC(t,\PT{y})}
		&= \sum_{\ell=0}^{\infty}\sum_{m=-\ell}^{\ell}e^{-2\freigv t}\vTc\:\shY(\PT{x}) \shY(\PT{y})\\
		&= \sum_{\ell=0}^{\infty}e^{-2\freigv t}\vTc(2\ell+1)\Legen{\ell}(\PT{x}\cdot\PT{y}),
	\end{align*}
	where the second equality uses the addition theorem in \eqref{eq:addition.theorem}. This means that the covariance $\expect{\solC(t,\PT{x}) \solC(t,\PT{y})}$ is a zonal function and is thus rotationally invariant. Thus, $\solC(t,\cdot)$ is a $2$-weakly isotropic Gaussian random field on $\sph{2}$.
	
	On the other hand, for $\ell,\ell'\in \Nz$, $m,=-\ell,\dots,\ell$ and $m'=-\ell',\dots,\ell'$, by the $2$-weak isotropy of $\RF_{0}$, Lemma~\ref{lem:orth.Fcoe.RF} and \eqref{eq:sol.fr.stoch.Cauchy},
	\begin{align*}
		\expect{\Fcoe{\solC(t)}\Fcoe[\ell' m']{\solC(t)}}
		&= e^{-\freigv t}e^{-\freigv[{\eigvm[\ell']}] t}\:\expect{\Fcoe{(\RF_{0})}\Fcoe[\ell' m']{(\RF_{0})}}\\
		&= e^{-2\freigv t} \vTc\: \Kron\Kron[m m'],
	\end{align*}
	thus proving \eqref{eq:Fcoe.sol.Cauchy.cov}.
\end{proof}

\begin{proof}[Proof of Theorem~\ref{thm:sol.fr.SPDE.scalar}] One can rewrite the equation \eqref{eq:fSPDE} as
\begin{equation*}
 \sol(t) = \solC(t_{0}) - \int_{0}^{t}\psi(-\LBo) \sol(u) \IntD{u} + \fBmsph(t).
\end{equation*}
Then by Definition~\ref{defn:fBmsph},
\begin{align}\label{eq:sol.shY}
	&\sum_{\ell=0}^{\infty}\sum_{m=-\ell}^{\ell}\InnerL{\sol(t),\shY}\shY\notag\\
	&\qquad = \sum_{\ell=0}^{\infty}\sum_{m=-\ell}^{\ell} \biggl(
		\InnerL{\solC(t_{0}),\shY} - \int_{0}^{t}\InnerL{\sol(u),\shY}\psi(-\LBo)\IntD{u} +
		\KLcoe(t)\biggr)\shY\notag\\
	&\qquad = \sum_{\ell=0}^{\infty}\sum_{m=-\ell}^{\ell} \biggl(
		\InnerL{\solC(t_{0}),\shY} - \freigv\int_{0}^{t}\InnerL{\sol(u),\shY} \IntD{u} +
		\KLcoe(t)\biggr) \shY.
\end{align}
By the uniqueness of the spherical harmonic representation, see e.g. \citep{Rudin1950}, solving \eqref{eq:sol.shY} is equivalent to solving the equations
\begin{equation*}
	\InnerL{\sol(t),\shY} =
		\InnerL{\solC(t_{0}),\shY} - \freigv\int_{0}^{t}\InnerL{\sol(u),\shY} \IntD{u} + \KLcoe(t)
\end{equation*}
 for $m=-\ell,\dots,\ell$, $\ell\in \Nz$.
Using the variation of parameters, we can solve this integral equation (of $t$) path-wise to obtain, for $m=-\ell,\dots,\ell$, $\ell\in \Nz$,
\begin{equation*}
	\InnerL{\sol(t),\shY} =
		e^{-\freigv t}\InnerL{\solC(t_{0}),\shY}  + \int_{0}^{t}e^{-\freigv (t-u)}\IntD{\KLcoe(u)},
\end{equation*}
see e.g. \cite{ChKaMa2003,LaSc2015}.
Using \eqref{eq:fBmsph.real.expan},
\begin{align*}
	\sol(t)
	&= \sum_{\ell=0}^{\infty}\sum_{m=-\ell}^{\ell}\Bigl(e^{-\freigv t} \InnerL{\solC(t_{0}),\shY}  + \int_{0}^{t}e^{-\freigv (t-u)}\IntD{\KLcoe(u)}
	\Bigr)\shY\\
	&= \sum_{\ell=0}^{\infty}\biggl(\sum_{m=-\ell}^{\ell}e^{-\freigv (t+t_{0})}\Fcoe{(\RF_{0})}\shY\\
	&\hspace{1.5cm} + \sqrt{\vfBm}\Bigl(\int_{0}^{t}e^{-\freigv(t-u)}\IntBa[{\ell 0}](u)\:\shY[\ell,0]\notag\\
          &\hspace{3cm} +\sqrt{2}\sum_{m=1}^{\ell}\bigl(\int_{0}^{t}e^{-\freigv(t-u)}\IntBa(u) \:\CRe \shY \\
          &\hspace{4.6cm}+ \int_{0}^{t}e^{-\freigv(t-u)}\IntBb(u) \:\CIm \shY\bigr)\Bigr)\biggr),
\end{align*}
where the second equality uses \eqref{eq:sol.fr.stoch.Cauchy}, which completes the proof.
\end{proof}

\begin{proof}[Proof of Proposition~\ref{prop:Xt.coeff.fSI}] 
	For $\hurst=1/2$, It\^{o}'s isometry, see e.g. \citep[Lemma~3.1.5]{Oksendal2003} and \citep{LaSc2015}, gives
\begin{align}
	\expect{\left|\int_{s}^{t}e^{-\freigv(t-u)}\IntB(u)\right|^{2}}
	&= \expect{\int_{s}^{t}\left|e^{-\freigv(t-u)}\right|^{2}\IntD{u}}\notag\\\
	&= \int_{s}^{t}e^{-2\freigv(t-u)}\IntD{u}\notag\\
	&= \frac{1-e^{-2\freigv(t-s)}}{2\freigv}\notag\\[1mm]
	&= \bigl(\GSD[\ell,t-s]^{\hurst}\bigr)^{2}.\label{eq:BM.int.var1}
\end{align} 
	
For $1/2<\hurst<1$, by \citep[Eq.~1.3]{MeMiVa2001},
\begin{align}
	&\expect{\left|\int_{s}^{t}e^{-\freigv(t-u)}\IntB(u)\right|^{2}}\notag\\
	&\quad= \hurst(2\hurst-1)\int_{s}^{t}\int_{s}^{t}e^{-\freigv(2t-u-v)}|u-v|^{2\hurst-2}\IntD{u}\IntD{v}\notag\\
	&\quad= \hurst(2\hurst-1)\left(\int_{2s}^{s+t}
	\int_{0}^{x-2s}e^{-\freigv(2t-x)}y^{2\hurst-2}\IntD{y}\IntD{x}
	+ \int_{s+t}^{2t}
	\int_{0}^{2t-x}e^{-\freigv(2t-x)}y^{2\hurst-2}\IntD{y}\IntD{x}\right)\notag\\
	&\quad= \hurst\left(\int_{2s}^{s+t}
	e^{-\freigv(2t-x)}(x-2s)^{2\hurst-1}\IntD{x}	+ \int_{s+t}^{2t}
	e^{-\freigv(2t-x)}(2t-x)^{2\hurst-1}\IntD{x}\right)\notag\\
	&\quad= \hurst\left(\int_{0}^{t-s}
	e^{-\freigv(2t-2s-u)}u^{2\hurst-1}\IntD{u}	+ \int_{0}^{t-s}
	e^{-\freigv u}u^{2\hurst-1}\IntD{u}\right)\notag\\
	&\quad= \hurst(t-s)^{2\hurst}\left(e^{-2\freigv(t-s)}\int_{0}^{1}
	e^{\freigv(t-s)u}u^{2\hurst-1}\IntD{u}	+ \int_{0}^{1}
	e^{-\freigv (t-s)u}u^{2\hurst-1}\IntD{u}\right)\notag\\
	&\quad=\hurst \Gamma(2\hurst) (t-s)^{2\hurst} \left(e^{-2\freigv(t-s)}\igamma{2\hurst,-\freigv(t-s)} + \igamma{2\hurst,\freigv(t-s)}\right),\label{eq:fBm.int.var1}
\end{align}
where the second equality uses integration by substitution $x=u+v$ and $y=u-v$. 

By e.g. \citep[p.~253]{PiTa2000}, \eqref{eq:BM.int.var1} and \eqref{eq:fBm.int.var1}, the fractional stochastic integral in \eqref{eq:fBm.int} is a Gaussian random variable with mean zero and variance $\bigl(\GSD[\ell,t-s]^{\hurst}\bigr)^{2}$ given by \eqref{eq:var.int.fBm} (and \eqref{eq:var.int.BM}) for $\hurst\in[1/2,1)$.

The upper bound in \eqref{eq:Gvar.UB} is by \citep[Theorem~1.1]{MeMiVa2001}:
\begin{align*}
		\Gvar[\ell,t-s]
		&=\expect{\left|\int_{s}^{t}e^{-\freigv(t-u)}\IntB(u)\right|^{2}}\\
		&\le C_{\hurst}\left(\int_{s}^{t}\left|e^{-\freigv(t-u)}\right|^{\frac{1}{\hurst}}\IntD{u}\right)^{2\hurst}\\
		&\le C_{\hurst} (t-s)^{2\hurst},
\end{align*}
thus completing the proof.
\end{proof}

\begin{proof}[Proof of Proposition~\ref{prop:Gvar.diff.t.fBm}] We first consider for $\hurst=1/2$. For $\ell=0$, the statement immediately follows from $\GSD[0,t]=\sqrt{t}$. For $\ell\ge1$, it follows from \eqref{eq:var.int.BM} that
\begin{equation}\label{eq:Gvar.BM.increment}
	\left|\GSD[\ell,t+h]-\GSD[\ell,t]\right|
	=\Bigl(\sqrt{1-e^{-2\freigv (t+h)}} - \sqrt{1-e^{-2\freigv t}}\Bigr)
		\sqrt{\frac{1}{2\freigv}}.
\end{equation}

When $t=0$, the formula \eqref{eq:Gvar.BM.increment} with the mean-value theorem gives as $h\to0+$ that there exists $h_1\in (0,h)$ such that
\begin{equation*}
	\left|\GSD[\ell,t+h]-\GSD[\ell,t]\right|
	=\sqrt{1-e^{-2\freigv h}}
		\sqrt{\frac{1}{2\freigv}}\le e^{-\freigv h_{1}} h^{1/2}\le h^{1/2}.
\end{equation*}
In a similar way, when $t>0$ and $h\to0+$, there exists $t_1\in (t,t+h)$ such that
\begin{equation*}
	\left|\GSD[\ell,t+h]-\GSD[\ell,t]\right|
	= \sqrt{\frac{\freigv }{2(1-e^{-2\freigv t_{1}})}}\:e^{-2\freigv t_{1}} h
	\le \sqrt{\frac{\freigv }{2(1-e^{-2\freigv t})}}\:e^{-2\freigv t} h.
\end{equation*}

For $1/2<\hurst<1$,
\begin{align*}
	&\int_{0}^{t+h}e^{-\freigv (t+h-u)} \IntB(u)\\
	&\quad= e^{-\freigv h}\int_{0}^{t}e^{-\freigv (t-u)} \IntB(u)
	   + \int_{t}^{t+h}e^{-\freigv (t+h-u)} \IntB(u).
\end{align*}
This with the triangle inequality for $\Lpprob{2}$ gives
\begin{align*}
	&\left|\GSD[\ell,t+h]-\GSD[\ell,t]\right|\\
	&\quad= \left|\norm{\int_{0}^{t+h}e^{-\freigv (t+h-u)} \IntB(u)}{\Lpprob{2}} - \norm{\int_{0}^{t}e^{-\freigv (t-u)} \IntB(u)}{\Lpprob{2}}\right|\\
	&\quad\le \norm{\int_{0}^{t+h}e^{-\freigv (t+h-u)} \IntB(u)-\int_{0}^{t}e^{-\freigv (t-u)} \IntB(u)}{\Lpprob{2}}\\
	&\quad= \norm{\left(e^{-\freigv h}-1\right)\int_{0}^{t}e^{-\freigv (t-u)}\IntB(u) + \int_{t}^{t+h}e^{-\freigv (t+h-u)} \IntB(u)}{\Lpprob{2}}\\	
	&\quad\le \left|1-e^{-\freigv h}\right|\norm{\int_{0}^{t}e^{-\freigv (t-u)}\IntB(u)}{\Lpprob{2}} \\
	&\qquad+ \norm{\int_{t}^{t+h}e^{-\freigv (t+h-u)} \IntB(u)}{\Lpprob{2}}.
\end{align*}
This with \eqref{eq:Gvar.UB} and the mean-value theorem gives that as $h\to0+$, there exists $h_{2}\in(0,h)$ such that
\begin{equation*}
	\left|\GSD[\ell,t+h]-\GSD[\ell,t]\right|
	\le C_{\hurst}\left(\freigv e^{-\freigv h_{2}} h\: t^{\hurst} +  h^{\hurst}\right)
	\le C_{\hurst}\left(\freigv h^{1-\hurst} t^{\hurst} +  1\right)h^{\hurst},
\end{equation*}
thus completing the proof.
\end{proof}

\begin{proof}[Proof of Theorem~\ref{thm:trsol.err}] The proof views the solution at given time $t$ as a random field on the sphere and uses an estimate of the convergence rate of the truncation errors of a $2$-weakly isotropic Gaussian random field on $\sph{2}$.

Let $\sola(t):=\sum_{\ell=0}^{\infty}\trsola(t)$ and $\solb(t):=\sum_{\ell=0}^{\infty}\trsolb(t)$. Proposition~\ref{prop:Xt.coeff.fSI} with \citep[Theorem~5.13]{MaPe2011} shows that for $t\ge0$, $\solb(t)$ is a $2$-weakly isotropic Gaussian random field with angular power spectrum $\{\vfBm \Gvar\}_{\ell\in\Nz}$. By \eqref{eq:var.int.BM} and \eqref{eq:feigv.est} for $\hurst=1/2$ and by \eqref{eq:Gvar.UB} for $1/2<\hurst<1$,
	\begin{equation*}
		\sum_{\ell=0}^{\infty}\vfBm \:\Gvar \ell^{2\smind+1}
		\le C_{\alpha,\gamma,t}\sum_{\ell=0}^{\infty} \vfBm  (1+\ell)^{2\smind+1} <\infty.
	\end{equation*}
	This and \citep[Corollary~4.4]{LeSlWaWo2017} imply $\solb(t)\in \sob{2}{\smind}{2}$ $\Pas$ Then, \citep[Propositions~5.2]{LaSc2015}
	gives
	\begin{equation}\label{eq:sol.b.err}
		\normB{\solb(t)-\sum_{\ell=0}^{\trdeg}\trsolb(t)}{\Lppsph{2}{2}} \le C \trdeg^{-\smind} ,
	\end{equation}
	where the constant $C=C_{\smind} \sqrt{\var{\normb{\solb(t)}{\sob{2}{\smind}{2}}}}$ depends on a constant $C_{\smind}$ and the standard deviation of the Sobolev norm of $\solb(t)$, where $C_{\smind}$ depends only on $\smind$.
	
	On the other hand,
	\begin{align}\label{eq:sola.L2err}
	\normB{\sola(t)- \sum_{\ell=0}^{\trdeg}\trsola(t)}{\Lppsph{2}{2}}
		& = \norm{\sum_{\ell=L+1}^{\infty}\sum_{m=-\ell}^{\ell}
		e^{-\freigv t}\InnerL{\solC(t_{0}),\shY}\shY}{\Lppsph{2}{2}}\notag\\
		&\le e^{-\freigv[{\eigvm[\trdeg]}] t} \norm{\solC(t_{0})}{\Lppsph{2}{2}}\notag\\
		&\le C \trdeg^{-\smind} \norm{\solC(t_{0})}{\Lppsph{2}{2}},
	\end{align}
	where the second line uses that $\freigv$ is increasing with respect to $\ell$, see \eqref{eq:fr.eigvm}, and in the last inequality, the constant $C$ depends only on $\alpha,\gamma$ and $t$, and we used \eqref{eq:feigv.est}.
	This with \eqref{eq:sol.b.err} gives \eqref{eq:sol.tr.err}.
\end{proof}

\begin{remark}
	In the proof of Theorem~\ref{thm:trsol.err}, the $L_{2}$-error in \eqref{eq:sol.b.err} for $\solb(t)$ which is driven by the fBm $\fBmsph(t)$ is the dominating error term.	The constant $C$ in \eqref{eq:sol.b.err} depends on the standard deviation of the Sobolev norm of $\solb(t)$. This implies that Theorem~\ref{thm:trsol.err} only needs the condition on the convergence rate of the variances $\vfBm$ of the fBm (but does not need the condition on the initial random field $\RF_{0}$).

The constant $C$ in \eqref{eq:sola.L2err} can be estimated by
\begin{equation*}
	C\ge \max_{\trdeg\ge1}\frac{L^{\smind}}{e^{\freigv[{\eigvm[\trdeg]}]t}}.
\end{equation*}
This implies
\begin{equation*}
	C\ge e^{-C'}(C'/t)^{C'},
\end{equation*}
where $C':=r/(\alpha+\gamma)$ and we used \eqref{eq:feigv.est}. This shows that when time $t\to0+$, the constant $C$ in \eqref{eq:sola.L2err} is not negligible.
\end{remark}

\begin{proof}[Proof of Lemma~\ref{lem:sol.represent}]
	For $t\ge0$, $h>0$, by \eqref{eq:trsol},
	\begin{align}\label{eq:trsol.inc}
  \trsol[\ell](t+h) &=  e^{-\freigv h}\trsol[\ell](t) \notag\\
  &\hspace{1.1cm}+ \sqrt{\vfBm}\Bigl(\int_{t}^{t+h}e^{-\freigv(t+h-u)}\IntBa[{\ell 0}](u)\:\shY[\ell,0]\notag\\
          &\hspace{2.5cm} +\sqrt{2}\sum_{m=1}^{\ell}\bigl(\int_{t}^{t+h}e^{-\freigv(t+h-u)}\IntBa(u) \:\CRe \shY\notag\\
          &\hspace{4.1cm}+ \int_{t}^{t+h}e^{-\freigv(t+h-u)}\IntBb(u) \:\CIm \shY\bigr)\Bigr).
\end{align}
By Proposition~\ref{prop:Xt.coeff.fSI}, for $m=-\ell,\dots,\ell$, $i=1,2$,
	 \begin{equation*}
	 \int_{t}^{t+h}e^{-\freigv(t+h-u)}\IntB(u) \sim \normal{0}{\Gvar[\ell,h]},
	 \end{equation*}
	 where $\Gvar[\ell,h]$ is given by \eqref{eq:var.int.fBm}.
	 Then \eqref{eq:trsol.inc} can be written as
	 	\begin{align*}
  \trsol[\ell](t+h)
   &=  e^{-\freigv h}\trsol[\ell](t) \\
  &\hspace{1cm}+ \sqrt{\vfBm}\:\GSD[\ell,h]\Bigl(\incsola[\ell 0](t) \:\shY[\ell 0] \\
  &\hspace{3.2cm}+ \sqrt{2} \sum_{m=1}^{\ell}\bigl(\incsola(t)\:\CRe\shY
	+ \incsolb(t)\:\CIm\shY\bigr)\Bigr),
\end{align*}
where $\{(\incsol[\ell m]^{1}(t),\incsol[\ell m]^{2}(t))| m=-\ell,\dots,\ell, \ell\in\Nz\}$ is a sequence of independent and standard normally distributed random variables. This and \eqref{eq:U.l} give \eqref{eq:variat.sol.time}.
\end{proof}

\begin{proof}[Proof of Theorem~\ref{thm:UB.fBm.L2err.sol}] By \eqref{eq:variat.sol.0} and \eqref{eq:var.int.BM},
	\begin{equation*}
		\trsol[\ell](t+h) - \trsol[\ell](t) = \bigl(e^{-\freigv (t+h)} - e^{-\freigv t}\bigr) \trsol[\ell](0)
		+ \sqrt{\vfBm}(\GSD[\ell,t+h]-\GSD)
		 \incsol(0).
	\end{equation*}
	Then,
	\begin{align*}
		&\sol(t+h)-\sol(t) \\
		&\quad=
		\sum_{\ell=0}^{\infty}\bigl(\trsol[\ell](t+h) - \trsol[\ell](t)\bigr)\\
		&\quad= \sum_{\ell=0}^{\infty}\bigl(e^{-\freigv (t+h)} - e^{-\freigv t}\bigr) \trsol[\ell](0) + \sum_{\ell=0}^{\infty}(\GSD[\ell,t+h]-\GSD) \sqrt{\vfBm}\:\incsol(0).
	\end{align*}
	Taking the squared $\Lp{2}{2}$-norms of both sides of this equation with Parseval's identity gives
	\begin{align}\label{eq:variat.sol.L2}
		&\normb{\sol(t+h)-\sol(t)}{\Lp{2}{2}}^{2}\notag\\
		&\quad= \sum_{\ell=0}^{\infty}\sum_{m=-\ell}^{\ell}\bigl(e^{-\freigv (t+h)} - e^{-\freigv t}\bigr)^{2}\bigl|\Fcoe{\solC(t_{0})}\bigr|^{2}\notag\\
		&\qquad + \sum_{\ell=0}^{\infty}\sum_{m=-\ell}^{\ell}(\GSD[\ell,t+h]-\GSD)^{2}\vfBm \bigl|\Fcoe{\incsol(0)}\bigr|^{2}\notag\\
		&\quad\le h^{2}\sum_{\ell=0}^{\infty}\sum_{m=-\ell}^{\ell}e^{-2\freigv t_2}\bigl|\Fcoe{\solC(t_{0})}\bigr|^{2}
		 + C h^{2\hurst} \sum_{\ell=0}^{\infty}\sum_{m=-\ell}^{\ell}(1+\freigv)^{2}\vfBm
		 \bigl|\Fcoe{\incsol(0)}\bigr|^{2}\notag\\
		&\quad\le h^{2}\sum_{\ell=0}^{\infty}\sum_{m=-\ell}^{\ell}\bigl|\Fcoe{\solC(t_{0})}\bigr|^{2}
		 + C h^{2\hurst} \sum_{\ell=0}^{\infty}\sum_{m=-\ell}^{\ell}(1+\freigv)^{2}\vfBm
		 \bigl|\Fcoe{\incsol(0)}\bigr|^{2},
	\end{align}
	where the first inequality uses Corollary~\ref{corol:Gvar.diff.t.fBm} and the mean value theorem for the function $f(t):=e^{-\freigv t}$ and $t_{2}$ is a real number in $(t,t+h)$.
	
	By \eqref{eq:variat.sol.L2}, \eqref{eq:Fcoe.incsol.cov}, \eqref{eq:feigv.est} and Propositions~\ref{prop:isotr.sol.Cauchy} and \ref{prop:Gvar.diff.t.fBm}, the squared mean quadratic variation of $\sol(t+h)$ from $\sol(t)$ is, as $h\to0+$,
	\begin{align*}\label{eq:variat.sol.L2psph}
		&\expect{\normb{\sol(t+h)-\sol(t)}{\Lp{2}{2}}^{2}}\\
		&\quad\le C h^{2\hurst} \biggl(\sum_{\ell=0}^{\infty}\sum_{m=-\ell}^{\ell}\expect{\bigl|\Fcoe{\solC(t_{0})}\bigr|^{2}}
		 + \sum_{\ell=0}^{\infty}\sum_{m=-\ell}^{\ell}(1+\freigv)^{2}\vfBm\expect{\bigl|\Fcoe{\incsol(0)}\bigr|^{2}}\biggr)\\
		&\quad\le C h^{2\hurst} \sum_{\ell=0}^{\infty}(2\ell+1)\left(\vT + (1+\freigv)^{2}\vfBm\right)\\
		&\quad= C h^{2\hurst},
	\end{align*}
	where the constant $C$ in the last line depends only on $\alpha$, $\gamma$, $t$, $\vT$ and $\vfBm$. This completes the proof.
\end{proof}

\end{document}